\documentclass[10pt]{amsart}

\usepackage{latexsym,amsfonts, amssymb, verbatim}
\usepackage[english]{babel}

\numberwithin{equation}{section}

\newtheorem{propo}{Proposition}[section]

\newtheorem{lemma}[propo]{Lemma}
\newtheorem{corol}[propo]{Corollary}
\newtheorem{theo}[propo]{Theorem}
\newtheorem*{maintheo}{Main Theorem}

\theoremstyle{definition}
\newtheorem{defi}[propo]{Definition}
\newtheorem*{notation}{Notation}
\newtheorem*{rem}{Remark}
\newtheorem{remar}[propo]{Remark}

\newcommand{\Aut}{\mathop{\rm Aut}\nolimits}

\newcommand{\al}{\alpha}
\newcommand{\diag}{\mathop{\rm diag}\nolimits} 
\newcommand{\Id}{\mathop{\rm Id}\nolimits}
\newcommand{\ch}{\mathop{\rm char}\nolimits}
\newcommand{\lan}{\langle}
\newcommand{\ran}{\rangle}
\newcommand{\lam}{\lambda }

\newcommand{\Gal}{\mathop{\rm Gal}\nolimits}

\newcommand{\equad}{\quad\textrm{and}\quad}

\def\F{\mathbb{F}}
\def\SL{\mathrm{SL}}
\def\GL{\mathrm{GL}}
\def\PSL{\mathrm{PSL}}
\def\PGL{\mathrm{PGL}}

\def\SU{\mathrm{SU}}
\def\PSU{\mathrm{PSU}}

\def\GU{\mathrm{GU}}
\def\PGU{\mathrm{PGU}}

\def\Sp{\mathrm{Sp}}
\def\PSp{\mathrm{PSp}}

\def\SO{\mathrm{SO}}
\def\PSO{\mathrm{PSO}}

\def\POmega{\mathrm{P}\Omega}
\def\Mat{\mathrm{Mat}}

\begin{document}

\title[Almost cyclic  elements in cross-characteristic
representations]{Almost cyclic  elements in cross-characteristic
representations of finite groups of Lie type}

\author{Lino Di Martino}
\email{lino.dimartino@unimib.it}
\address{Dipartimento di Matematica e Applicazioni, Universit\`a degli Studi di Milano-Bicocca,
Via R. Cozzi 55,  20125 Milano, Italy}

\author{Marco A. Pellegrini}
\email{marcoantonio.pellegrini@unicatt.it}
\address{Dipartimento di Matematica e Fisica, Universit\`a Cattolica del Sacro Cuore\\
Via Musei 41, 25121 Brescia, Italy}

\author{Alexandre E. Zalesski}
\email{alexandre.zalesski@gmail.com}
\address{Academy of Sciences of Belarus, Minsk, 
Department of Physics, Mathematics and Informatics,
Prospekt Nezalejnasti 66, Minsk, 220000, Belarus}


\begin{abstract}
This paper is a significant contribution to a general programme aimed to classify all projective irreducible 
representations 
of finite simple groups over an algebraically closed field, in which  the
 image of at least one element is represented by
an almost cyclic matrix (that is, a square matrix $M$ of size $n$ over a field $\F$ with the property that
there exists $\al\in \F$ such that $M$ is similar to
 $\diag(\al\cdot \Id_k, M_1)$, where $M_1$ is cyclic and $0\leq k\leq n$).
While a previous paper dealt with the Weil representations of
finite classical groups, which play a key role in the general picture, the present paper provides a conclusive answer
for all cross-characteristic projective irreducible representations of the finite quasi-simple groups of Lie type and
their automorphism groups.
\end{abstract}

\maketitle

\section{Introduction}

Let $V$ be a vector space of finite dimension $n$ over an arbitrary field $\F$, and let $M$ be a square matrix of size 
$n$ over $\F$. Then 
$M$ is said to be cyclic if the characteristic polynomial and the minimum
polynomial of $M$ coincide. Note that a matrix $M\in \Mat_n(\F)$ is cyclic if and only if the
$\F\langle M \rangle$-module $V$ is cyclic, that is, is generated by
a single element. This is  standard terminology in module theory, and
the source of the term `cyclic matrix'. Matrices with simple spectrum
often arising in applications are cyclic.  We consider a generalisation of the notion of cyclic matrix, namely, we 
define a matrix $M\in \Mat_n(\F)$ to be almost cyclic if
there exists $\al\in \F$ such that $M$ is similar to
 $\diag(\al\cdot \Id_k, M_1)$, where $M_1$ is cyclic and $0\leq k\leq n$.

Examples of almost cyclic matrices arise naturally in the study  
of matrix groups over finite fields. For instance, if an element 
$g\in \GL(V)$ acts irreducibly on $V/V'$, where $V'$ is some eigenspace of 
$g$ on $V$,  then $g$ is almost cyclic. Transvections and  reflections 
 are other examples, as well as unipotent matrices with
Jordan  form consisting of a single non-trivial block.

The aim of the present paper is to implement a crucial step within a general project, which can be stated as follows: 
determine all irreducible subgroups $G$ of $\GL(V)$ which are generated by almost cyclic matrices, mainly when $\F$ is 
algebraically closed. The motivations for such a programme  (including the connections with the area of linear group 
recognition) have been described in full detail in \cite{DMZ12}, to which paper we also refer for the citations of all
relevant past results connected to our subject.

As current applications seem to
focus on $p$-elements, we will limit ourselves to the study of elements $g\in G$ of any $p$-power order, $p$ a prime. 
Furthermore, we will make a systematic use of representation theory and will take for granted the 
classification theorem of finite simple groups. In view of it, the non-abelian finite simple groups consist of 
$26$ sporadic groups, the alternating groups of degree $>4$, and the so-called finite groups of Lie type. The latter 
come together with their defining characteristics, and in this paper it will be convenient for us to distinguish two 
groups of Lie type of distinct defining characteristic even if they are isomorphic.
For instance, we view $\PSL_3(2)$ and $\PSL_2(7)$ as distinct.

The sporadic simple groups and their covering groups have been completely dealt with in  \cite{DMPZ}. 
In \cite{DMZ10},  we started to deal with finite groups of Lie type, and determined all the  irreducible representations
of a quasi-simple group of Lie type $G$ over an algebraically closed field $\F$ of characteristic coprime to the 
defining characteristic of $G$, in which the image of at least one unipotent
element $g$ is represented by an almost cyclic matrix.  The complementary case, when $g$ is unipotent in $G$, and the 
characteristic of $\F$ is the defining characteristic of $G$, has been settled for classical groups in \cite{Su13}, and 
for the exceptional groups of Lie type in \cite{TeZ}.  
This leaves open the case  when $g$ is a semisimple element of prime-power order of $G$. The case where $g^2=1$ already
attracted attention many years ago, and was treated in \cite{PH, Wag, ZS1, ZS80}.  
 
We should also mention that the occurrence of almost cyclic $p$-elements in irreducible $\F$-representation of finite
simple groups, where $p= \ell=\ch \F$, has been studied recently in \cite{Cr}.

In  \cite{DMZ12}, the focus was on Weil representations of
finite classical groups (there, an ample overview of these representations is given in Section 5.1).
The reason to deal first with Weil representations was the strong evidence that most examples of
semisimple almost cyclic elements occur in Weil representations, and the actual occurrence of such elements 
was thoroughly examined in \cite{DMZ12}.

In the present paper, we will consider the occurrence of almost cyclic elements of prime-power order in projective
cross-characteristic irreducible representations $\phi$ of finite simple classical groups $L$ which are not Weil 
(Sections \ref{LUSpodd}, \ref{low}, \ref{otherclass}). Furthermore, we will also consider the group extensions $G$ of
groups $L$ by outer automorphisms (note that a separate ad hoc section (Section \ref{PSL2}) is required for the case
when $\PSL_2(q)\subseteq G\subseteq  \Aut  \PSL_2(q)$). Finally, we will also deal with the projective irreducible 
representations of finite
simple exceptional groups of Lie type and their automorphism groups (Section \ref{sec exc}).

We note explicitly that, for the sake of completeness, we choose to consider as simple groups of Lie type 
the commutator subgroups of the groups $\PSp_4(2)$, $G_2(2)$, ${}^2 F_4(2)$ and ${}^2 G_2(3)$.
The main result of the present paper is summarised in the following:

\begin{maintheo}\label{main}
Let $L$ be a finite simple group of Lie type, let $g \in \Aut L$ be a $p$-element for some prime $p$, and let 
$G=\langle 
L,g \rangle$. Let $\phi$ be an irreducible projective representation of $G$ over an algebraically closed 
field $\F$ of characteristic $\ell$, different from the defining characteristic of $L$. Suppose that $\phi$ is 
non-trivial on $L$; furthermore, assume that $g^2\neq 
1$, $L\neq \PSL_2(q)$ and $\phi(g)$ is almost cyclic. Then either
$\phi_{|L}$ is an irreducible Weil representation of $L$, where $L\in \{\PSL_n(q)\;  (n\geq 3),\; \PSU_n(q) \; (n\geq 
3),\; \PSp_{2n}(q)\;
(n\geq 2,\, q \textrm{ odd} )\}$ and $g$ is semisimple, or $L$ belongs to the following list of groups:
$$\begin{array}{r}
\mathcal{L} = \{\PSL_3(2), \PSL_3(4), \PSL_4(2), \PSU_3(3), \PSU_3(4), \PSU_4(2),\PSp_4(3), \PSU_5(2),\\
 \PSp_4(2)', \PSp_4(4),\PSp_6(2), \POmega_8^+(2), {}^2B_2(2^3),G_2(2)', G_2(3),G_2(4), {}^2G_2(3)'\}.
  \end{array}$$
\end{maintheo} 

A more detailed statement, providing full information on the groups $L$ belonging to $\mathcal{L}$, is given in  
Theorem \ref{main2}. The case when $g$ is an  involution is fully dealt with in 
Lemma \ref{news2a} for classical groups, and in Lemma \ref{except}  for the groups of exceptional type.

One of our main tools is Lemma \ref{uu3}, which gives, for an arbitrary finite group $G$, an upper bound for $\dim\phi$ in terms of the order of $g\in G$, provided $g$ is
almost cyclic. For many pairs $(\phi,g)$ this upper bound conflicts with the known lower bound
for $\dim\phi$, which violates the existence of the pair in question.

In order to apply that lemma to a group $G$, where $L \subseteq G \subseteq \Aut L$, and $L$ is a simple group of Lie 
type, we need an accurate upper bound for $\eta_p(\Aut L)$, the exponent of a
Sylow $p$-subgroup of $\Aut L$, and for $\alpha(g)$, the minimal number of conjugates of $g$ generating $G$. Results on 
$\eta_p(\Aut
L)$ are collected in Section \ref{se3}. Rather precise information on $\alpha(g)$ is available in a paper by Guralnick
and Saxl \cite{GS}, which is our main reference for this matter.  If $g\in L$ and $C_L(g)$ contains no unipotent
elements, then $\al(g)\leq 3$ (see Lemma \ref{g12}, deduced from a result of Gow \cite{Gow}). However, 
if $\dim\phi$ is relatively small in comparison with $\eta_p(G)$, much more analysis is required.  
This actually happens either if $|L|$ is small or if  $G=L$ and $\phi$ is a Weil representation.  (As already mentioned, 
the  Weil representations have been dealt with in \cite{DMZ12}). 

In our analysis we will deal separately with the following cases: 
(i) $|g|=2$ (see Lemma \ref{news2a} and Section \ref{sec exc}); 
(ii) $L=\PSL_2(q)$ (see Section \ref{PSL2});
(iii) The groups $\PSL_n(q)$, $n>2$; $\PSU_n(q)$, $n >2$; $\PSp_{2n}(q)$, $n > 1$, $q$ odd (see Section \ref{LUSpodd});
(iv) Some exceptional low-dimensional classical groups (see Section \ref{low}); 
(v) The groups $\Sp_{2n}(q)$, $n > 1$, $q$ even; $\Omega_{2n+1}(q)$, $n> 2$, $q$ odd; $\POmega_{2n}^{\pm}(q)$, $n> 
3$ (see 
Section \ref{otherclass}); 
 (vi) The exceptional groups of Lie type (see Section \ref{sec exc}).

\begin{notation}
The centre of a group $G$ is denoted by $Z(G)$. Let $G$ be a finite group. We write $|G|$ for the order of $G$. For $g\in G$, $|g|$ is the order of $g$, and $o(g)$ is the order of $g$ modulo $Z(G)$, that is, the minimum integer $k$ such that $g^k\in Z(G)$. 
 If $g$ does not lie in a proper normal subgroup of $G$, we write $\alpha(g)$
for the minimum number of conjugates of $g$ that generate $G$.  For a prime $p$ we denote by $\eta_p(G)$
the exponent of a Sylow $p$-subgroup of $G$.

A finite field of order $q$ (where $q$ is a prime-power) is denoted by $\F_{q}$. 
Let $V$ be a vector space of finite dimension $m>1$ over the finite field $\F_{q}$. 
The general linear group $\GL(V)$, that is, the group of all bijective linear transformations of $V$, and the special 
linear group $\SL(V)$, that is, the subgroup of $\GL(V)$ consisting of the transformations of determinant $1$,  are 
denoted by $\GL_m(q)$ 
and $\SL_m(q)$, respectively. 

Suppose that the space $V$ is endowed with a non-degenerate orthogonal, 
symplectic or unitary form. Then $I(V)$ denotes the group of the isometries of $V$, that is the set of all elements
$x\in\GL(V)$ preserving the form. Then we will loosely 
use the term `finite classical group' for a subgroup $G$ of $I(V)$ containing the commutator subgroup $I(V)^{\prime}$, 
as well as the projective image of $G$, that is, the quotient of $G$ modulo its scalars, which will be denoted by 
$\mathrm{P}G$. 
In particular:  
if $V$ is a symplectic space over  $\F_{q}$,  $I(V)$ will be denoted by $\Sp_m(q)$; if $V$ is a unitary 
space over the field $\F_{q^{2}}$, $I(V)$ will be denoted by $\GU_m(q)$; and if $V$ is an orthogonal 
space over  $\F_{q}$, $I(V)$ will be denoted by $\mathrm{O}_m(q)$. More precisely, if  $V$ is an orthogonal 
space of even dimension $m$, we will denote $I(V)$ by $\mathrm{O}_m^+(q)$ or $\mathrm{O}_m^-(q)$ if $V$ has maximal 
or non-maximal 
Witt index, respectively. 
Accordingly, the subgroups of the unitary and orthogonal isometry groups consisting of the elements of determinant $1$ 
will be denoted by $\SU_m(q)$, $\SO_m(q)$, $\SO_m^+(q)$ and $\SO_m^-(q)$, respectively. 
We recall that $\Sp_m(q)' =\Sp_m(q)$, unless $m = 2$ and $q\leq 3$, or $(m,q) = (4,2)$, whereas 
$\GU_m(q)' = 
\SU_m(q)$, unless $m = 2$ and $q\leq 3$, or $(m,q) = (3,2)$. In the orthogonal case, the groups $\SO_m(q)$, 
$\SO_m^+(q)$ and $\SO_m^-(q)$ contain $I(V)'$ as a subgroup of index $2$. We will denote by  $\Omega_m(q)$ and 
$\Omega_m ^{\pm }(q)$ the commutator 
subgroups of the corresponding orthogonal groups, and accordingly, we will refer to
a classical group as a subgroup of $I(V)$ containing the commutator subgroup $I(V)'$. 
Furthermore, we will denote by $\mathrm{CSp}_{2n}(q)$ the conformal symplectic group, that is, the group of 
symplectic similarities. Also, it should be noted that in places the term `classical group' 
will be meant to include also the groups $\GL_m(q)$ and $\SL_m(q)$ (considering $V$ endowed with the identically 
zero bilinear form). It is well-known that, for $m>2$, or $m=2$ and $q>2$, one has  $\GL_m(q)' = \SL_m(q)$.

If $G$ is a finite group of Lie type of defining characteristic $r$, we assume throughout the paper that $\ell$ is 
coprime to $r$. 
Furthermore, note that: (1) the conjugacy classes of elements of $G$ are labelled according to the GAP package 
\cite{GAP}; (2) the extensions $G$ of a finite simple group of Lie type $L$ by an outer automorphism will be labelled 
according to the Atlas notation as given in \cite{Atl} and \cite{MAtl}, unless we wish to indicate explicitly that $G$ 
is the extension of $L$ by the group of field automorphisms, in which case $G$ will be denoted, say, by 
$\mathrm{P\Sigma L}_m(q)$ or
$\mathrm{P\Sigma U}_m(q)$ for $L= \PSL_m(q)$ or $L=\PSU_m(q)$, respectively.

By a projective $\F$-representation of an arbitrary group $G$ we mean a  homomorphism $\phi: G\to \PGL_m(\F)$ for some 
$m$. 
Let $\bar G$ denote the pre-image of $\phi(G)$ in $\GL_m(\F)$ and let $\psi$ denote the linear representation of $\bar 
G$ 
arising in this way from $\phi$. If $g\in G$, and $\bar g$ denotes a pre-image of $\phi(g)$ in
$\GL_m(\F)$, we will say for short that $\phi(g)$ is almost cyclic if so is $\psi(\bar g)$ in
$\GL_m(\F)$. Note that, clearly, the property of $\phi(g)$ of being almost cyclic does not depend on the choice of 
$\bar g$.
\end{notation}

\section{Preliminaries}

We recall the following basic definition:

\begin{defi} Let $M$ be an $(n\times n)$-matrix over an arbitrary field $\F$.
We say that $M$ is almost cyclic  if there exists $\al\in \F$ such that $M$ is similar to
 $\diag(\al\cdot \Id_k, M_1)$, where $M_1$ is cyclic and $0\leq k\leq n$.
\end{defi}

It is trivial to observe that every $(3\times 3)$-matrix over any
field $\F$ is almost cyclic. 
Another elementary  observation, which will be useful throughout the paper,  is the following:
if $M\in \GL(V)$ is almost cyclic, and $U$ is
an $M$-stable subspace of $V$, then the induced actions of  $M$ on $U$ and
on $V/U$ yield almost cyclic matrices.

Throughout the paper, unless  stated otherwise, we assume that $\F$ is an algebraically closed field 
of characteristic  $\ell$. We emphasise that the representations of $G$ we will consider are all over $\F$.

\subsection{Almost cyclicity and dimension bounds}

Let $G$ be a finite group and $\phi$ be a projective irreducible $\F$-representation of $G$. Let $g \in G$ and suppose 
that 
$\phi(g)$ is almost cyclic and non-scalar.  
Our results will be essentially based on the fact that $\dim\phi$ can be bounded in terms of the minimum number
$\al(g)$ of conjugates of $g$ sufficient to generate $G$.  In this subsection we explain this in detail.

The following crucial results on generation by conjugates are due to Guralnick
and  Saxl \cite{GS}.

\begin{propo}\label{GS1} 
Let $L$ be a simple group of Lie type, and let
$1 \neq  x \in  \Aut L$.  Denote by $\alpha(x)$ the minimum number
of $L$-conjugates of $x$ sufficient to generate  $\lan x,L\ran$. Then the
following holds:
\begin{itemize}
\item[(1)] {\rm \cite[Theorem 4.2]{GS}} Let $L$ be a simple classical group, 
and assume that the natural module for $L$ 
has  dimension $n \geq  5$. Then $\alpha(x) \leq  n$, unless $L = \PSp_n(q)$ with
$q$ even, $x$ is a transvection and $\alpha(x) = n+ 1$.
\item[(2)] {\rm \cite[Theorem 5.1]{GS}} Let $L$ be a simple  exceptional
group of Lie type, 
of untwisted Lie rank $m$. Then $\alpha(x)\leq  m + 3$, except possibly for the
case $L = F_4(q)$ 
with $x$ an involution, where $\alpha(x) \leq  8$.
\end{itemize}
\end{propo}

\begin{propo}\label{GS2}
Under the same assumptions and of {\rm Proposition \ref{GS1}}, and
assuming additionally that $x$ has prime order, the following holds:
\begin{itemize}
\item[(1)] {\rm \cite[Lemma 3.1]{GS}}  Let $L=\PSL_2(q)$,  with $q \geq 4$.
Then $\al(x) \leq 3$, unless either
{\rm (i)} $x$ is a field automorphism of order $2$ and $\al(x) \leq 4$, except
that $\al(x) = 5$ for $q=9$; or
{\rm (ii)} $q = 5$, $x$ is a diagonal automorphism of order $2$ and $\al(x) =4$.
Moreover, if $x$ has odd order, then $\al(x) = 2$, unless $q=9$, $|x|=3$ and
$\al(x) = 3$. 
\item[(2)] {\rm \cite[Lemma 3.2]{GS}} If $L=\PSL_3(q)$, then $\al(x)\leq 3$
unless
$x$ is an involutory graph-field automorphism and $\al(x)\leq 4$.
\item[(3)] {\rm \cite[Lemma 3.3]{GS}} If $L=\PSU_3(q)$, $q>2$,
then $\al(x)\leq 3$, unless $q=3$ and $x$ is an inner involution with 
$\al(x)=4$.
\item[(4)] {\rm \cite[Theorem 4.1(c),(d)]{GS}} If $L=\PSL_4(q)$, then 
$\al(x)\leq 4$, unless one of the following holds:
{\rm (i)} $q>2$, $x$ is an involutory graph automorphism and $\al(x)\leq 6$;
{\rm (ii)} $q=2$, $x$ is an involutory graph automorphism and $\al(x)\leq 7$.
\item[(5)] {\rm \cite[Lemma 3.4]{GS}} If $L=\PSU_4(q)$, then $\al(x)\leq 4$,
unless one of the following holds:
{\rm (i)} $x$ is an involutory graph automorphism and $\al(x)\leq 6$;
{\rm (ii)} $q=2$ with $x$ a transvection and $\al(x)\leq 5$.
\item[(6)] {\rm \cite[Theorem 4.1(f)]{GS}} If $L=\PSp_4(q)$,
then $\al(x)\leq 4$, unless $x$ is an involution and $\al(x)\leq 5$, or $q=3$
and $\al(x)\leq 6$.
\end{itemize}
\end{propo}

The above quoted bounds on $\alpha(x)$ play a role in the following lemma, which is crucial for our purposes.

 \begin{lemma}\label{uu3} 
Let $H$ be a finite group, let  $\phi: H\to \PGL_m(\F)$ be an irreducible projective $\F$-representation of $H$, where 
$\F$ is an arbitrary
algebraically closed field, and let $\bar H$ be the pre-image of $\phi(H)$ in $\GL_m(\F)$. Let $M$ be the module 
affording the linear representation $\psi$  of $\bar H$ arising from $\phi$. For any $h\in H$ denote by $\bar h$ a 
pre-image of $\phi(h)$ in $\bar H$, and choose $h \in H$ such that $H$ is generated some conjugates of $h$. Let 
$d=\dim(\lam\cdot\Id - \bar h)M\neq 0$
for some $\lam\in F$. Then $\dim \phi \leq d\al(h)$. In particular, if $\phi(h)$ is almost cyclic, then $\dim \phi \leq \al(h)(o(h)-1)\leq \al(h)(|h|-1)$.
\end{lemma}

\begin{proof}
Let  $m=\al(h)$ and $h_1, \ldots, h_m$ be conjugates of $h$ generating $H$. Let $M_i=  (\lambda\cdot \Id -h_i)M$. Then
every $\bar h_i$ acts scalarly as $\lambda$ on $M/M_i$, and hence it acts scalarly as $\lambda$ on any quotient $M/U$ where $U$ is a subspace containing  $M_i$, thus in particular on $M/(\sum M_i)$. This is true for 
every $i$, so every $\bar h_i$ acts as $\lambda$ on  $M/(\sum M_i)$.
Since $\lan h_1,\ldots,h_m\ran=H$, the group $\bar H$ acts scalarly on $M/(\sum M_i)$. It follows that $\sum M_i$ is a
non-zero $\bar H$-submodule of $M$, and hence coincides with $M$. Since $\dim(\sum^m_{i=1}  
(\lam\cdot\Id - \bar h_i)M)\leq \sum^m_{i=1}  
\dim (\lam\cdot\Id -\bar h_i)M= m\cdot \dim (\lam\cdot\Id -\bar h)M=md$, the first part of the statement follows. The 
second part  is proven in \cite[Lemma 2.1]{DMPZ}.
\end{proof}

We remark here that throughout the paper we will frequently make use of Lemma \ref{uu3} without explicit reference to it.
The following results concerning regular semisimple elements will be useful for our purposes.

\begin{lemma}\label{d102}\cite{Gow}
Let $G$ be a simple group of Lie type of characteristic $\ell$ and let
$1 \neq g\in G$ be a regular semisimple element. Then every semisimple element of $G$
can be expressed as a product $ab$, where $a,b$ are conjugate to $g$.
\end{lemma}

As a corollary of  Lemma \ref{d102}, we obtain a small refinement of case (1) of Proposition \ref{GS2}, when considering
$2$-elements:

\begin{lemma}\label{87t} 
Let $L=\PSL_2(q)$, where $q$ is odd and $q>9$. Let $g\in L$ be a $2$-element such that $g^2\neq 1$. Then   
$\alpha(g)=2$.
\end{lemma}

\begin{proof} 
Note that $g$ is a regular semisimple element of $L$
(as $C_{L}(g)$ contains no unipotent element). So, by Lemma \ref{d102}, there exists $x\in L$   conjugate to $g$ such that 
$gx=s$, where $s$ is a Singer cycle of $L$.
If  $q>9$, it is well-known that every proper subgroup of $L$ containing $s$ is either $S=\lan s\ran$ or $N=N_{L}(S)$, a
dihedral group of order $2|S|$.
Therefore, the result is true unless both $g,x\in N$. But the latter is not the case.  Indeed, suppose that $g,x\in N$.
As $N$ is dihedral, all the  $2$-elements of order $\geq4$ of $N$ are in $S$.
Therefore $g\in S$ and $x\in S$.  
 As $S$ is cyclic and $|g|=|x|$, it follows readily that $|gx|<|g|$; this obviously contradicts $gx=s$, as claimed. 
\end{proof}
 
\begin{lemma}\label{g12}\cite[Lemma 2.8]{DMZ12}
Let  $G$ be a simple group of Lie type.  Then the following holds: 
\begin{itemize}
\item[(1)]  $G$ is generated by two semisimple elements.
\item[(2)] Let $g\in G$ be any regular semisimple element. Then  $\al(g)\leq3$, that is, 
$G$ can be generated by three elements conjugate to $g$. 
\end{itemize}
\end{lemma}

The following two lemmas are crucial for us in view of the results above. The first one yields the best known lower
bounds for the dimension of an  irreducible projective cross-characteristic $\F$-representation of a finite classical 
group. The second one
provides lower bounds for the dimension of any non-Weil irreducible projective cross-characteristic $\F$-representation 
of a finite
classical group.

\begin{lemma}\cite[Table 1]{Ho}\label{md6}
Let $L$ be a simple classical group and let $\phi$
be an irreducible projective cross-characteristic $\F$-representation of $L$, with  $\dim\phi>1$. 
Then the following holds:
\begin{itemize}
\item[(1)] If $L=\PSL_2(q)$, then $\dim\phi\geq \frac{q-1}{(2,q-1)}$, unless $q = 4,9$; in the latter exceptional 
cases,
$\dim\phi\geq 2, 3$, respectively.
\item[(2)]  If $L=\PSL_n(q)$, $n>2$, then $\dim\phi\geq \frac{q^n-1}{q-1}-2$, unless $(n,q)= (3,2),(3,4),$ 
$(4,2),(4,3)$;
in the latter exceptional cases, $\dim\phi\geq 2, 4, 7, 26$, respectively.
\item[(3)] If $L=\PSU_n(q)$, $n>2$, $(n,q)\neq(4,2),(4,3),$ then $\dim\phi\geq \frac{q^n-q}{q+1}$ if $n$ is odd, 
while $\dim\phi\geq \frac{q^n-1}{q+1}$ if $n$ is even; if $(n,q) = (4,2),(4,3)$, then  $\dim\phi\geq 4, 6$, 
respectively.
\item[(4)] If $L=\PSp_{2n}(q)$, $n>1$, and $q$ is odd,  then $\dim\phi\geq \frac{q^n-1}{2}$.
\item[(5)] If $L=\PSp_{2n}(q)$, where $n>1$, $q$ is even, and $(n,q)\neq (4,2)$, then $\dim\phi\geq
\frac{q(q^n-1)(q^{n-1}-1)}{2(q+1)}$. If $L=\PSp_4(2)'$, then $\dim\phi\geq 2$.
\item[(6)] Let $L=\POmega_{2n+1}(q)$, where $n>2$ and $q$ is odd.
If $q=3$, and $n\neq 3$,  
then $\dim\phi\geq \frac{(3^{n}-1)(3^{n}-3)}{3^2-1};$ if $q>3$, then $\dim\phi\geq\frac{q^{2n}-1}{q^2-1}-2$; if
$(n,q)=(3,3)$, then $\dim\phi\geq 27$.
\item[(7)]  Let $L=\POmega_{2n}^+(q)$, $n>3$. If $q<4$ and $(n,q)\neq (4,2)$, then $\dim\phi\geq 
\frac{(q^n-1)(q^{n-1}-1)}{q^2-1}$; if $q\geq4$, then $\dim\phi\geq  \frac{(q^n-1)(q^{n-1}+q)}{q^2-1} - 2$; if
$L=\POmega_8^+(2)$, then $\dim\phi\geq 8$.
\item[(8)] Let $L=\POmega_{2n}^-(q)$, $n>3$, where $(n,q)\neq (4,2), (4,4), (5,2),(5,3)$. Then $\dim\phi\geq 
\frac{q(q^{n}+1)(q^{n-2}-1)}{q^2-1}-1;$ for the exceptional cases, $\dim\phi\geq 32, 1026, 151,$ $2376$, respectively.
\end{itemize}
\end{lemma}

From \cite[Theorems 2.1 and  2.7]{GMST}, \cite[Theorem 1.1]{GT2},   \cite{GT} and \cite[Theorem 1.6]{HM} 
we obtain the following:

\begin{lemma}\label{gmst2}
Let $L\in \{\PSL_n(q),\,n \geq 3;\;\PSU_n(q),\, n\geq 3;\; \PSp_{2n}(q),n\geq 2, q \textrm{ odd}\}$  and 
let $\phi$ be a non-trivial cross-characteristic projective irreducible representation of $L$. 
Then either $\phi$ is a Weil representation, or one of the following occurs.
\begin{itemize}
\item[($A$)]  If $L=\PSL_n(q)$ and $(n,q)\neq (3,2), (3,4), (4,2), (4,3), (6,2), (6,3)$, then  $\dim \phi \geq 
\vartheta^+_n(q)$, where
\begin{equation*}\label{dsl}
\vartheta^+_n(q)=
\begin{cases}
(q-1)(q^2-1)/\gcd(3,q-1) & \textrm{ if } n=3,\\
(q-1)(q^3-1)/\gcd(2,q-1) & \textrm{ if } n=4, \\
(q^{n-1}-1)((q^{n-2}-q)/(q-1)-\kappa_{n-2}) & \textrm{ if  } n\geq 5.
\end{cases}
\end{equation*}
Here, $\kappa_n=1$ if $\ell$ divides $(q^n-1)/(q-1)$ and $0$ otherwise.
Furthermore, if $(n,q)=(4,3), (6,2), (6,3)$, then   $\dim \phi \geq 
\vartheta^+_n(q)$, where
$\vartheta^+_n(q) = 26, 61,362$, respectively.
\item[($B_1$)]  If $L=\PSU_n(q)$, with $n\geq 5$ and $(n,q)\neq (6,2)$, then  $\dim \phi \geq \vartheta^-_n(q)$, where
 \begin{equation*}\label{dsu}
\vartheta^-_n(q)=
\begin{cases}
\frac{q^{n-2} (q-1)(q^{n-2}-q)}{q+1} & \textrm{ if } n \textrm{  is odd},\\
\frac{q^{n-2} (q-1)(q^{n-2}-1)}{q+1} & \textrm{ if } n \textrm{  is even}.\\
\end{cases}
\end{equation*}
\item[$(B_2)$] If $L=\PSU_4(q)$ and $q>3$, then $\dim \phi\geq \vartheta^-_{4}(q)$, where  
$$\vartheta^-_{4}(q)=
\begin{cases}
\frac{(q^2+1)(q^2-q+1)-2}{2} & \textrm{ if } q \textrm{  is odd},\\
(q^2+1)(q^2-q+1)-1 & \textrm{ if } q \textrm{  is even}.\\
\end{cases}$$
\item[$(B_3)$] If  $L=\PSU_3(q)$ and $q\geq 5$, then $\dim \phi \geq \vartheta^-_{3}(q)$, where
$$\vartheta^-_{3}(q)=
\begin{cases}
\frac{(q-1)(q^2+3q+2)}{6} & \textrm{ if } \gcd(3,q+1)=3,\\
\frac{2q^3-q^2+2q-3}{3} &  \textrm{ if } \gcd(3,q+1)=1.
\end{cases}$$
\item[$(C)$] If $L=\PSp_{2n}(q)$ and $(2n,q)\neq (6,2)$, then $\dim \phi\geq \sigma_{n}(q)$, where
$$\sigma_{n}(q)=\frac{(q^n-1)(q^n-q)}{2(q+1)}.$$
\end{itemize}
\end{lemma}

\subsection{The case $|g|=2$}
 
In this subsection we consider the case $|g|=2$, which yields a significant part of the examples in which almost cyclic
elements do occur. 
   
\begin{lemma}\label{news2a}
Let $L$ be a finite simple classical group and let $G=\langle L,g\rangle$, where $g\in \Aut L$ is an 
involution. Let $\phi$ be a cross-characteristic projective irreducible $\F$-representation  of $G$ which is non-trivial 
on $L$.
Then $\phi(g)$ is almost cyclic if and only if one of the following holds:
\begin{itemize}
\item[(1a)] $G=\PSL_2(4)$,  $\dim \phi=2,3$ and $\ell\neq 2$;
\item[(1b)] $G=\mathrm{P\Sigma L}_2(4)$ and either $\dim \phi= 4$ and $\ell\neq 2,5$ or $\dim \phi=2,3$ and $\ell=5$;
\item[(2a)] $G=\PSL_2(5)$, $\dim \phi=2,3$ and $\ell\neq 5$;
\item[(2b)] $G=\PGL_2(5)$, $\dim \phi= 4$ and $\ell\neq 5$;
\item[(3)] $G=\PSL_2(7)$, $\dim\phi =3$ and $\ell\neq 7$;
\item[(4a)] $G=\PSL_2(9)$,  $\dim \phi = 3$ and $\ell \neq 2,3$;
\item[(4b)] $G=\mathrm{P\Sigma L}_2(9)$ and either $ \dim \phi=5$ and $\ell\neq 2,3$ or $\dim \phi=4$ and $\ell=2$;
\item[(5)] $G=\PSL_3(2)$ and either $\dim \phi =3$ and $\ell\neq 2$ or $\dim \phi=2$ and $\ell=7$;
\item[(6)] $G=\mathrm{P\Sigma L}_3(4)$, $\dim \phi=4$ and $\ell=3$;
\item[(7)] $G=\PSL_4(2).2$, $\dim \phi=7$, $\ell\neq 2$ and $g$ yields a graph automorphism;
\item[(8a)] $G=\PSU_4(2)$, $\dim \phi=4,5$ and $\ell\neq 2$ ({class $2a$}); 
\item[(8b)] $G=\mathrm{P\Sigma U}_4(2)$ and either $\dim \phi=6$ and $\ell\neq 2,3$  or $\dim \phi=5$ and $\ell=3$;
\item[(9)] $G=\PSU_4(3).2_2$, $\dim \phi=6$ and $\ell\neq 3$;
\item[(10a)] $G=\PSp_4(3)$ and either $\dim \phi=4$ and $\ell\neq 3$ (class $2a$)
or $\dim \phi=5$ and  $\ell\neq 2,3$ (class $2a$);
\item[(10b)] $G=\mathrm{PCSp}_4(3)$, $\dim \phi=6$ and $\ell\neq 2,3$;
\item[(11)] $G=\PSp_6(2)$, $\dim \phi=7$ and $\ell\neq 2$ (class $2a$);
\item[(12)] $G=\PSO_8^+(2)$, $\dim \phi=8$ and $\ell\neq 2$.
\end{itemize}
\end{lemma}

\begin{proof}
(I) Suppose that $\phi(g)$ is almost cyclic.  By Lemma \ref{uu3}, this implies that
\begin{equation}\label{bound2}
\dim \phi\leq \alpha(g)(|g|-1)=\alpha(g).
\end{equation}
An upper bound for the value of $\alpha$ is given by Propositions \ref{GS1} and \ref{GS2}, while a lower bound
for the value of $\dim \phi$ is given in Lemma \ref{md6}.

For the sake of brevity, we give full details of the computations yielding the statement of the lemma only in the case where
$L = \PSL_n(q)$. The other classical groups can be dealt with in a similar way.

Let $L=\PSL_2(q)$, with $7\leq q\neq 9$. Then \eqref{bound2} yields $\frac{q-1}{(q-1,2)}\leq 4$, which holds only
if $q=7$.
If $L\in \{\PSL_2(4),\PSL_2(5)\}$, then $\dim \phi \geq 2$ and $\alpha(g)=3$ if $g \in L$, $\alpha(g)\leq 4$ 
otherwise.
If $L=\PSL_2(7)$, then $\dim \phi\geq 3$ and $\alpha(g)=3$.
If $L=\PSL_2(9)$, then $\dim \phi \geq  3$ and $\alpha(g)=3$, unless $g \not \in L$, in which case $\alpha(g)\leq 5$.

Let  $L=\PSL_3(q)$, with $q\neq 2,4$. In this case \eqref{bound2} yields  $q^2+q-1\leq 4$, which is false.
If $L=\PSL_3(2)$, then $\dim \phi\geq 2$ and $\alpha(g)=3$.
If $L=\PSL_3(4)$, then $\dim \phi \geq 4$ and $\alpha(g)=3$, unless $g \not \in L$, in which case $\alpha(g)=4$.

Let $L=\PSL_4(q)$, with $q\geq 4$. Then \eqref{bound2} yields $q^3+q^2+q-1\leq 6$, which is clearly impossible.
If $L=\PSL_4(2)$, then $\dim \phi\geq 7$ and $\alpha(g)\leq 6$, unless $g \not \in L$, in which case $\alpha(g)\leq 
7$.
If $L=\PSL_4(3)$, then $\dim \phi\geq 26$, whence a contradiction.

Finally, let $L=\PSL_n(q)$, with $n\geq 5$. Then  \eqref{bound2} yields
$\frac{q^n-1}{q-1}-2\leq n$, a contradiction.

By the computations above, we readily get items (1a) to (7) of the statement.
\smallskip

\noindent (II) Conversely, in all the occurrences listed in the statement of the lemma (items (1) to (12)), there 
actually exists
at least one representation $\phi$ of $G$ and a conjugacy class of involutions $g$ such that $\phi(g)$ is  almost
cyclic. These $\phi$'s and these $g$ can be readily determined using \cite{Atl,MAtl} and the GAP package.
\end{proof}

\begin{remar}
The outer automorphism group of the group $L=\PSL_2(9)$ is elementary abelian of order $4$. Let $g$ be an outer 
automorphism of $L$ of order $2$. Then $g$ is either a field automorphism, or a diagonal automorphism, or the product of 
the two. In the first case  $G=\langle L,g\rangle = \mathrm{P\Sigma L}_2(9)$, in the second case $G=\langle 
L,g\rangle = \PGL_2(9)$, in the third case, $G=\langle L,g\rangle$ is isomorphic to $M_{10}$, a maximal subgroup of 
the Mathieu group $M_{11}$, a sporadic simple group. It is also well-known that $\PSL_2(9) \cong\Sp_4(2)' 
\cong \Omega_4^-(3)$, and $\Sp_4(2) \cong \mathrm{P\Sigma L}_2(9)$.
If we view $L $ as $\Sp_4(2)'$ and let $G=\langle L,g\rangle$, where $g\in \Aut L$ is an 
involution, then the above lemma misses the $\F$-representations of $G$ for $\ell=3$. So,  let $\phi$ be a 
cross-characteristic projective irreducible $\F$-representation  of $G$ which is non-trivial on $L$.
Then direct inspection shows that $\phi(g)$ is almost cyclic if and only if one of the following occurs:
\begin{itemize}
\item[(1)] $\dim\phi=3$, $G= \Sp_4(2)'$ and $\ell \neq 2$;
\item[(2)] $\dim\phi=3$, $G= \PGL_2(9)$ and $\ell = 3$;
\item[(3)] $\dim\phi=4$, $G=\Sp_4(2)$, and $\ell=3$;
\item[(4)] $\dim\phi=5$, $G= \Sp_4(2)$, and $\ell \neq 2,3$.
\end{itemize}
\end{remar}

\section{Maximum order of a $p$-element in $\Aut L$}\label{se3}
 
Let $L$ be one of the following groups: $\PSL_n(q)$, where $n>1$ and $(n,q)\neq (2,2), (2,3)$;  
$\PSU_n(q)$, where $n>2$ and $(n,q)\neq(3,2)$; $\PSp_{2n}(q)$, where $n>1$ and $(n,q)\neq(2,2)$. So $L$ is simple. 
Let $r$ be a prime,  and $q=r^\alpha$. 

Recall that every automorphism of $L$ is a product of inner, diagonal, field and graph automorphisms. Let $  A_d $ be 
the subgroup of $ \Aut L$ generated by all  inner and diagonal automorphisms. It is well-known that $A_d$ is normal in 
$\Aut L$. Moreover, $A_d\cong \PGL_n(q)$ or $\PGU_n(q)$ if $L=\PSL_n(q)$ or $\PSU_n(q)$, respectively. 
Thus, 
$|A_d/L|= (q-1,n)$ if $L=\PSL_n(q)$, and $(q+1,n)$ if $L=\PSU_n(q)$. If $ L = \PSp_{2n}(q)$, then $A_d\cong 
\mathrm{PCSp}_{2n}(q)$. Thus,  
$|A_d:L|=2$ for $  q$ odd, and $A_d=L$ for $q$ even.  The group  $\Aut L/A_d$ is abelian. Set $\Phi=\Gal(\F_q/\F_r)$ 
and 
$\Phi_2=\Gal(\F_{q^2}/\F_r)$. If  $L=\PSL_n(q) $, then $\Aut L/A_d\cong  C_2\times \Phi$, where 
$C_2$ is the cyclic group of order $2$. If $L=\PSU_n(q)$, then $\Aut L/A_d\cong \Gal(\F_{q^2}/\F_r)=\Phi_2$. 
If  $L =\PSp_{2n}(q)$, $n>2$, or $n=2$ for $q$ odd, then $\Aut L/A_d\cong \Phi$. 
The case $L=\PSp_4(q)$ with $q$ even 
is more complex. In \cite[Theorem 28]{St} an automorphism $\phi$ of $\Sp_4(q)$, $q$ even, is constructed, such that 
$\phi^2$ is a 
generator of $\Phi$.
 It follows that $\Aut \Sp_4(q)/\Sp_4(q)$ is a cyclic group of order $2|\Phi|\cong \Phi_2$. So, if $g\in \Aut L$ is a 
$2$-element and $|\Phi|_2=2^m$, then $g^{2^{m+1}}\in L$.   
 See \cite[Table 5.1.A]{KL} for details.

Let $p$ be a prime. For any positive integer $m$, we denote by $|m|_p$ the $p$-part of $m$, that is the highest power of $p$ dividing $m$.
Furthermore, for any positive integer $m$, if $p$ does not divide $m$, then $e=e_p(m)$ is defined to be the
minimum integer $i>0$ such that $p$ divides $m^i-1$ if $p>2$. If $p=2$, then $e_2(m)$ is defined to be $1$ if $4\mid (m-1)$
and $2$ if $4\mid (m+1)$. Note that $e \geq 1$, and if $e>1$, then $e< \frac{m^e-1}{m-1}$. 

\begin{lemma}\label{f4t}
Let $p$ be a prime.
\begin{itemize}
\item[(1)] $e_p(q)=e_p(q^{p^k})$. In particular: if $|\F^\times_{q^{p^k}}|_p>1$, 
then $|\F^\times_{q}|_p>1$. 
\item[(2)] Suppose that $p\mid (q-1)$ if $p>2$, and $4\mid (q-1)$ otherwise.
Let $k>0$ be an integer. Then $|\F^\times_{q^{p^k}}|_p=p^k\cdot |\F^\times_q|_p$,
equivalently, $|q^{p^k}-1|_p=p^k|q-1|_p$.
\end{itemize}
\end{lemma}

\begin{proof}
Fermat's little theorem states that $p$ divides $n^p-n$ for any integer $n$, and hence $p\mid (n-1)$ if and only if 
$p\mid
(n^p-1)$. So $p \mid (q^i-1)$ if and only if $p\mid ((q^{p^k})^i-1)$, whence (1). For (2), see \cite[Lemma 7.5]{EZ}.
\end{proof}

Recall that, for a group $X$ we denote by $\eta_p(X)$ the exponent of a Sylow $p$-subgroup  of $X$.
The following  lemma is well known: for instance, see \cite{Ol}.

\begin{lemma}\label{ge7}
Let $G=\GL_n(q)$, $p$ a prime dividing $|G|$, $(p,q)=1$.
If $e=e_p(q)$, then $\eta_p(G)=p^l\cdot |q^e-1|_p$, where $l$ is such that $p^le\leq n <p^{l+1}e$. 
 \end{lemma}

\begin{corol}\label{ce7}
Let $q=q_0^2$ be odd. Then $\eta_2(\GL_n(q))=2\eta_2(\GL_n(q_0))$.
\end{corol}

\begin{proof}
Suppose first that $n=2^t$, $t>0$. By assumption $4$ divides $q-1$, and by Lemma \ref{ge7}
$\eta_2(\GL_n(q))=2^t|q-1|_2$. If $4$ divides $q_0-1$, then $\eta_2(\GL_n(q_0))=2^t|q_0-1|_2$, whereas if $4$ divides
$q_0+1$, then $\eta_2(\GL_n(q_0))=2^t|q_0+1|_2$. As $|q_0^2-1|_2=2|q_0-1|_2$ in the former case and
$|q_0^2-1|_2=2|q_0+1|_2$ in the latter case, the result follows.  If $2^{t+1}>n>2^t$, then
$\eta_2(\GL_n(q))=\eta_2(\GL_{2^t}(q))=2\eta_2(\GL_{2^t}(q_0))=2\eta_2(\GL_n(q_0))$, whence the statement.
\end{proof}

\begin{lemma}\cite[Lemma 3.2]{ZalM}\label{s33} 
Let $p$ be an odd prime and $n>1$. Let $G=\SL_n(q)$, with $p\mid (q-1)$ or $G=\SU_n(q)$, with $p\mid (q+1)$.
Then $G$ has no irreducible $p$-element. 
\end{lemma}

\begin{lemma}\label{za4}
Let $q=q_0^{p^k}$, $k>0$, and  $G=G(q)\in \{\GL_n(q), \SL_n(q), \GU_n(q),$ $\SU_n(q), \Sp_{2n}(q), 
\Omega_{2n+1}(q), \Omega_{2n}^\pm(q)\}$. 
Let $\Psi$ be the group of field automorphisms of $G$ of order $p^k$, $p>2$, and set 
 $H=G\cdot \Psi$. Then the following holds:
\begin{itemize} 
\item[(1)] $\eta_p(H)= \max\{p^{(k-i)}\eta_p(G(q_0^{p^i})): 0\leq i \leq k\}$.
\item[(2)] If $p\mid q$, then $\eta_p(H)= p^k\cdot \eta_p(G)$.
\item[(3)] If $p\nmid q$, then $\eta_p(H)= \eta_p(G)$.
\item[(4)] If $p \nmid q$, then $\eta_p(L)=\eta_p(L\cdot \Psi)$, where
$L$ is the simple non-abelian composition factor of $G$.
\end{itemize}
\end{lemma} 

\begin{proof}
(1) The statement follows from \cite[Corollary 14]{Zav} and its immediate generalisation to any $G(q)$ as remarked in
\cite{V17}, 
taking into account the well known fact that, as $p$ is odd, $G(q_0)\subseteq G(q_0^{p})$ (see \cite[Table 4.5A]{KL}).\\
(2) The statement follows from (1) and the known fact that  $\eta_p(G(q_0^{p^i}))=\eta_p(G(q_0))$ for $i=1,\ldots, k$. 
E.g., 
for an even stronger statement, \cite[Lemma 2.32]{TZ3}.\\
(3) The statement follows from \cite[Lemma 3.10]{V17}. \\
(4) It is well known that $L$ contains a subgroup isomorphic to a Sylow $p$-subgroup of $G$, unless $n$ is a $p$-power,  
$L=\PSL_n(q)$
and $p$ divides $(n,q-1)$ or  $L=\PSU_n(q)$ and $p$ divides $(n,q+1)$. Let us consider the exceptional cases. By 
Lemma 
\ref{s33}, every $p$-element  of $\SL_n(q)$ and $\SU_n(q)$ is reducible. It follows that    $\eta_p(\SL_n(q))\leq 
\eta_p(\GL_{n/p}(q))$ 
and $\eta_p(\SU_n(q))\leq \eta_p(\GU_{n/p}(q))$. Moreover, there is an embedding 
$j: \GL_{n/p}(q)\to \SL_n(q)$ such 
that 
$j(\GL_{n/p}(q))$ contains no scalar matrix (except the identity), which in turn yields an embedding 
$j: \GL_{n/p}(q)\to \PSL_n(q)$. Thus $\eta_p(\PSL_n(q))=\eta_p(\GL_{n/p}(q))=\eta_p(\SL_n(q))$. Similarly, we have 
$\eta_p(\PSU_n(q))=\eta_p(\GU_{n/p}(q))=\eta_p(\SU_n(q))$.  By (3), $\eta_p( L\cdot \Psi)=\eta_p
(\SL_n(q)\cdot \Psi)=\eta_p(\SL_n(q))=\eta_p(L)$ for $L=\PSL_n(q)$, and similarly for $L=\PSU_n(q)$. 
\end{proof}

\begin{lemma}\label{za4bis} 
Suppose $q=q_0^{2^k}$ where $k>0$ and $q$ is odd. Let $G = G(q) \in \{\GL_n(q),$ $\SL_n(q),\PSL_n(q)\}$ and let $\Psi$ be the group of
field automorphisms of $G$ of order $2^k$. Set  $H=G\cdot \Psi$. Then  $\eta_2(H)= \eta_2(G)$.
  \end{lemma} 
  
\begin{proof}
By \cite[Corollary 14]{Zav}, $\eta_2(H)= \max\{2^{(k-i)}\eta_2(G(q_0^{2^i})): 0\leq i \leq k\}$. For $G = \GL_n(q)$, 
the
statement then follows by iterated application of Corollary \ref{ce7}. 
 Otherwise, the argument used to prove item (4) of Lemma \ref{za4} 
 remains valid for $p=2$. (As $k>0$, Lemma \ref{s33} remains valid for $n=p=2$.) So  
$\eta_2(\PSL_n(q))=\eta_2(\GL_{n/2}(q))=\eta_2(\SL_n(q))$. 
Again by Corollary \ref{ce7}, we get 
$\eta_2(\GL_{n/2}(q))= 2^{k-i}\eta_2(\GL_{n/2}(q_0^{2^{i}}))$  for every $i$.  It follows that $\eta_2(H) = 
\eta_2(\GL_{n/2}(q)) = \eta_2(G)$, as claimed.
\end{proof}

\begin{lemma}\label{nn9}
Let  $G=\GL_n(q)$,  $q$ odd, and suppose that  $2^t\leq n<2^{t+1}$. Then $\eta_2(G)=2^t|q-1|_2$ if  $4\mid (q-1)$ and 
$2^t|q+1|_2$ if  $4\mid (q+1)$.
\end{lemma}

\begin{proof}
By Lemma \ref{ge7}, $\eta_2(G)=|q^e-1|_22^l $, where $2^le\leq n<2^{l+1}e$. If $e=1$, then $t=l$, and the statement
follows.  Let $e=2$. Then
 $|q^2-1|_2=|q-1|_2|q+1|_2=2|q+1|_2$ and $2^{l+1} \leq n < 2^{l+2}$. So $l=t-1$, whence  the statement.
\end{proof}

\section{The groups $\PSL_2(q)$}\label{PSL2}

In this section we deal with the groups $L=\PSL_2(q)$.
Set $G = \lan L,g \ran$, where  $1\neq g\in \Aut L$ is a $p$-element, for some prime $p$. Note that here we may assume 
$g^2\neq 1$ in view of Lemma \ref{news2a} and $ g\in (G\setminus \PGL_2(q))$, as the case $g\in \PGL_2(q)$ is settled 
in 
\cite{DMZ12}. Furthermore, recall that $L=\PSL_2(q)$ has no graph automorphism, and therefore $\Aut L/\PGL_2(q)\cong
\Gal( \F_q/\F_r)$. As above, let $\Phi=\Gal( \F_q/\F_r)$ and $p^m=|\Phi|_p$.

\begin{theo}\label{t22} 
Let $L=\PSL_2(q)$, $q\geq 4$, and let $G = \lan L,g \ran$, where  $1\neq g\in \Aut L$ is a $p$-element, for some 
prime 
$p$ and $g \not \in \PGL_2(q)$. Let $\phi$ be a cross-characteristic irreducible projective $\F$-representation 
of $G$, which is non-trivial on $L$. 
Suppose that $g^2\neq 1$.   Then $\phi(g)$ is not almost cyclic unless one of the following holds:
\begin{itemize}
\item[(1)] $p = 3$, $q = 8$, 
$|g|=9$, and $\dim \phi=7,8$ if $\ell \neq 3$, whereas $\dim \phi=7$ if $\ell = 3$.
\item[(2)] $p=2$, $q = 4$, $|g|=4$ and either $\ell\neq 3,5$ and $\dim \phi=5$, or $\ell=3$ and $\dim
\phi=4$, or $\ell= 5$ and $\dim \phi=2,3,5$.
\item[(3)] $p=2$, $q = 9$ and one of the following holds:
\begin{itemize}
\item[(i)] If $\ell \neq 2,5$, then 
\begin{itemize}
\item[(a)] $G=L.2_1 = \mathrm{P\Sigma L}_2(9)$, $|g|=4$ and $\dim \phi=4,5$;
\item[(b)] $G=L.2_2 = \PGL_2(9)$, $|g|=8$ and $\dim \phi=6,8,9$;
\item[(c)] $G=L.2_3 = M_{10}$, $|g|=8$ and $\dim \phi=6,8,9$.
\end{itemize}
\item[(ii)] If $\ell=2$, then 
\begin{itemize}
\item[(a)] $G=L.2_1 = \mathrm{P\Sigma L}_2(9)$, $|g|=4$ and $\dim \phi=4$;
\item[(b)] $G=L.2_2 =  \PGL_2(9)$, $|g|=8$ and $\dim \phi= 8$;
\item[(c)] $G=L.2_3 = M_{10}$, $|g|=8$ and $\dim \phi=8$.
\end{itemize}
\item[(iii)] If $\ell=5$, then
\begin{itemize}
\item[(a)] $G=L.2_1 = \mathrm{P\Sigma L}_2(9)$, $|g|=4$ and $\dim \phi=4,5$;
\item[(b)] $G=L.2_2 =  \PGL_2(9)$, $|g|=8$ and $\dim \phi=6,8$;
\item[(c)] $G=L.2_3 = M_{10}$, $|g|=8$ and $\dim \phi=8$.
\end{itemize}
\end{itemize}
\end{itemize}
\end{theo}

\begin{proof}
Let $q=q_0^{p^m}$. By Lemma \ref{md6}(1), $\dim\phi\geq \frac{q-1}{(q-1,2)}=\frac{q_0^{p^m}-1}{(q_0-1,2)}$, unless $q = 4, 9$. 
Note that if  $g$ is semisimple, then $|g| \leq p^m |L|_p =p^m |(q_0^2)^{p^m}-1|_p=p^m |q_0^2-1|_p$ 
by Lemma \ref{f4t}(2).

Suppose that  $p>2$ (and hence $q\neq 4,9$). If $g$ is unipotent, the result follows from \cite[Lemma 4.13]{DMZ10}. 
So, assume that $g$ is semisimple. Then, $|g|\leq p^m(q_0+1)$ and, by Lemma \ref{GS2},  $\alpha(g)=2$.
If $\phi(g)$ is almost cyclic, the inequality $q_0^{p^m}-1\leq 2 (2,q_0-1)(p^m(q_0+1)-1)$ must hold by Lemma \ref{uu3}.
However, this happens only for $(q,p)=(8,3)$, which gives item (1). Indeed (see \cite{Atl,MAtl}),
the group $G=\Aut \PSL_2(8)$
has irreducible representations of degree $7,8$ over the complex numbers, in which the elements of order $9$ are almost 
cyclic. This implies that the elements of order $9$ are almost cyclic in every irreducible constituent of degree $>1$ 
of a 
Brauer reduction of any of these representations modulo any prime.  
If $\ell=3$, we only get constituents of degree $7$, whereas, if $\ell=7$  these are of degree $7$ and $8$. Note
that the case of $\ell=2$ is irrelevant here, but will appear for $G={}^2G_2(3)\cong \SL_2(8)\cdot 3$.

Now, suppose that $p=2$ and consider first the case when $q > 9$ is odd. By the above,
$|g|\leq 2^{m+1} (q_0+1)$.
Suppose that $\phi(g)$ is almost cyclic. Then
$\frac{q-1}{2} \leq \dim \phi \leq \alpha(g) (|g|-1)$ by  Lemma \ref{uu3}. Thus, as $g^2\neq1$, by Lemma \ref{87t}
we must have $\frac{q_0^{2^m}-1}{2} \leq 2(2^{m+1} (q_0+1) -1)$.
This inequality holds only if either $m=1$ and $q_0=3,5,7,9,11,13$,  or $m=2$ and $q_0=3$.
However, $\eta_2(\PSL_2(81))=16$, contradicting Lemma \ref{uu3}.
Similarly, we can rule out $q=121$ and $q =169$, since $\eta_2(\PSL_2(121))\leq 16$ and
$\eta_2(\PSL_2(169))\leq 16$.
So, we are left with the cases $q=25,49$.

Next, suppose that $q>4$ is even.
Then $\dim \phi\geq q-1$ and clearly $ |g|\leq 2^{m+1}$.
If $\phi(g)$ is almost cyclic, then 
$q-1 \leq \dim \phi \leq \alpha(g) (|g|-1)$. By Lemma \ref{GS2}(1), $\alpha(g)\leq 4$. Thus, since $g^2\neq 1$, 
$ q_0^{2^m}-1 \leq 4(2^{m+1}-1 )$, which holds if and only if 
$q_0=2$ and $m=1,2$, i.e. if and only if $q=4, 16$.

In conclusion, we are left with the cases where $q \in \{4, 9, 16, 25, 49
\}.$\\
1) Suppose $q=4$. Then $|g| = 4$, $\alpha(g) = 2$, and using the packages GAP and MAGMA we get item (2) of the
statement.\\
2) Suppose $q=9$. Then $|g| = 4,8$ and $\alpha(g) = 2$. If $\ell \neq 2$, using the packages GAP and MAGMA we get the
results recorded in the statement, item (3).  So, suppose $\ell = 2$. Assume first $|g| = 4$. Then $\phi(g)$ almost
cyclic implies $\dim \phi \leq 6$, which in turn implies $\dim \phi = 4, 6$. If  
$G = \mathrm{P\Sigma L}_2(9)$,  then $G$ has
a $2$-modular representation of degree $4$ in which  $\phi(g)$ is actually cyclic, whereas $3.G$ has a faithful 
$2$-modular
representation of degree $6$ in which $\phi(g)$ is almost cyclic.  There are no other occurrences for $|g| = 4$. Next,
assume that $|g| = 8$. Then $\phi(g)$ almost cyclic implies $\dim \phi \leq 14$. Then either $G = \PGL_2(9)$ or $G = 
M_
{10}$. In the first case there is a $2$-modular representation of $G$ of degree $8$ in which $ \phi (g)$ is almost
cyclic, and a faithful $2$-modular representation of $3.G$ of degree $6$ in which $ \phi (g)$ is almost cyclic. In the
latter case, there is a $2$-modular representation of $G$ of degree $8$ in which  $\phi (g)$ is almost cyclic, and two
faithful $2$-modular representations of $3.G$ of degree $6$ and $9$ respectively, in which $\phi (g)$ is almost cyclic.\\
3) Suppose $q = 16$. Then $|g| = 4,8$ and $\alpha(g) = 2$. Since $\dim \phi \geq 15$, $\phi(g)$
cannot be almost cyclic by Lemma \ref{uu3}. So this case is ruled out.\\
4) Suppose $q=25$. Then $|g| = 4,8$ and $\alpha(g) = 2$. Since $\dim \phi \geq 12$, by Lemma \ref{uu3} we may assume $|g|=8$,
However, $L.2_2=\mathrm{P\Sigma L}_2(25)$ does not contain elements of order $8$.
So, we are left to consider the projective representations of $L.2_1=\PGL_2(25)$ and of $L.2_3$.
Making use of GAP we see that if $\ell \neq 2$, then $\phi(g)$ is not almost cyclic.
On the other hand, if $\ell=2$, then $\dim \phi\geq 24$, and hence $\phi(g)$ is not almost cyclic. So this case is ruled
out.\\
5) Suppose $q=49$. Then  $|g| = 4,8,16$ and $\alpha(g) = 2$. Since $\dim \phi \geq 24$, again by Lemma \ref{uu3} we may assume $|g|=16$.
However, $L.2_2=\mathrm{P\Sigma L}_2(49)$ does not contain elements of order $16$.
So, we are left to consider the projective representations of $L.2_1=\PGL_2(49)$ and of $L.2_3$.
Making use of GAP we see that if $\ell \neq 2,3$ $\phi(g)$ is not almost cyclic.
On the other hand, if $\ell=2$, then $\dim \phi\geq 48$, and hence $\phi(g)$ is not almost cyclic. So, also this last
case is ruled out.
\end{proof}

 \section{The groups $\PSL_n(q)$, $n>2$; $\PSU_n(q)$, $n >2$; $\PSp_{2n}(q)$, $n > 1$, $q$ odd}\label{LUSpodd}

 In this section we consider the simple groups $L=\PSL_n(q)$, $n>2$, $\PSU_n(q)$, $n>2$ with
$(n,q)\neq (3,2)$, $\PSp_{2n}(q)$, $q$ odd.
\smallskip

 Let $1\neq g\in \Aut L$ be a $p$-element, set $G=\lan L,g\ran$ and let 
$\phi$ be a cross-characteristic projective irreducible representation of $G$ that is non-trivial on $L$. Observe that 
if the restriction of $\phi$ to $L$ has a
constituent which is not a Weil representation of $L$, a lower bound of $\dim\phi$ is provided by Lemma \ref{gmst2}. On
the other hand, the occurrence of almost cyclic elements in the Weil representations of the groups $L$ listed above is
thoroughly dealt with in our earlier paper \cite{DMZ12}. There, it is also shown that they extend to the subgroup 
$A_d$ of $ \Aut L$ generated by all  inner and diagonal automorphisms.
Now, suppose that $g\in A_d \setminus L$ and $\phi$ restricted to $L$ decomposes into the sum of, say, $t$ constituents which
are all Weil representations of $L$. By Clifford's theorem these constituents are all $g$-conjugates. However, they
extend to $A_d$. Hence $t=1$; in other words, $\phi$ itself is a Weil representation of $A_d$, and is disposed of in
\cite{DMZ12}.  In view of this, this section will be mainly concerned with answering the following cases:
(A) the case when  the restriction of $\phi$ to $L$  has a constituent which is not a Weil representation of $L$;
(B) the case when $g\notin A_d$ and all the constituents of the restriction of $\phi$ to $L$  are Weil representations 
of $L$.

 \subsection{Case $(\mathrm{A})$}
 
Apart from very few groups,  case (A) is answered by the following.

\begin{lemma}\label{nonW}
Let $L$ be one of the following simple groups:
\begin{itemize}
\item[(1)] $\PSL_n(q)$, where $n\geq 3$ and  $(n,q)\not \in \{ (3,2), (3,4), (4,2)\}$;
\item[(2)] $\PSU_n(q)$, where $n\geq 3$ and  $(n,q)\not \in \{(3,3), (3,4),(4,2),(4,3),(5,2) \}$;
\item[(3)] $\PSp_{2n}(q)$, where $n\geq 2$, $q$ is odd  and  $(n,q)\neq (2,3)$.
\end{itemize}
Let $1 \neq g\in \Aut L$ be a $p$-element for some prime $p$ and set $G=\langle L, g\rangle$. 
Let  $\phi$ be a cross-characteristic projective irreducible representation of $G$ that
is non-trivial on $L$. 
If the restriction of $\phi$ to $L$ has a constituent which is not a Weil representation of $L$, 
then $\phi(g)$ is not almost cyclic.
\end{lemma}

\begin{proof}
Suppose that $\phi(g)$ is almost cyclic. By Lemma \ref{uu3} the following bound must be met:
\begin{equation}\label{bou}
\dim \phi \leq \alpha(g)\cdot (|g|-1) 
\end{equation}
Since the restriction of $\phi$ to $L$ has a constituent which is not a Weil representation of $L$,
lower bounds for $\dim \phi$ are provided by Lemma \ref{gmst2}.
Upper bounds for $\alpha(g)$ are given in Propositions \ref{GS1} and \ref{GS2}.
Finally, by \cite[Theorem 2.16]{GMPS}  we have $|g|\leq \mu(L)$, where
$$\mu(L)=\left\{\begin{array}{cl}
\frac{q^n-1}{q-1} & \textrm{ if } L=\PSL_n(q),\\
q^{n-1}-1 &  \textrm{ if } L=\PSU_n(q) \textrm{ with } n \textrm{ odd and } q \textrm{ not prime},\\
q^{n-1}+q &  \textrm{ if } L=\PSU_n(q) \textrm{ with } n \textrm{ odd and } q \textrm{ prime},\\
q^{n-1}+1 &  \textrm{ if } L=\PSU_n(q) \textrm{ with } n \textrm{ even and } q>2,\\
4(2^{n-3}+1) &  \textrm{ if } L=\PSU_n(q) \textrm{ with } n \textrm{ even and } q=2,\\
\frac{q^{n+1}}{q-1} &  \textrm{ if } L=\PSp_{2n}(q).
                 \end{array}\right.$$
                 
Suppose first that $L=\PSL_n(q)$. If $n=3$ and $q\geq 5$, then $\alpha(g) \leq 4$.
The bound $\vartheta_3^+(q)\leq 4(q^2+q)$ only holds for $q=5,7,13$.
Next, assume that $n=4$ and $q\geq 4$. In this case, $\alpha(g)\leq 6$: the bound $\vartheta_4^+(q)\leq 6 
(q^3+q^2+q)$ only holds 
for  $q=4,5,7,9,11,13$.
Now, assume that $n\geq 5$ and $(n,q)\neq (6,2),(6,3)$. In this case, the inequality  $\vartheta_n^+(q)\leq n 
\frac{q^n-q}{q-1}$ only holds for $(n,q) = (5,2)$.

If $L=\PSL_3(3)$, then
$|g|\in \{2, 3, 4, 8, 13 \}$ and $\dim \phi\geq 16$.
If $|g|\leq 3$, then $\alpha(g)\leq 3$, whereas if $|g|\geq 4$, then $\alpha(g)=2$.
Thus we readily get a contradiction for $|g| \neq  13$.  Likewise, for $|g|=13$, using the GAP package we see that 
$\phi(g)$ is not almost cyclic.

If $L\in\{\PSL_4(3), \PSL_5(2), \PSL_6(2), \PSL_6(3) \} $ we use the following data to get a contradiction with the bound 
\eqref{bou}:
\begin{itemize}
 \item[$(i)$] $L=\PSL_4(3)$: $|g| \in \{ 2, 3, 4, 5, 8, 9, 13 \}$ and $\dim \phi\geq 26$.
If $|g|\geq 5$, then $\alpha(g)=2$; if $|g|\leq 4$ then $\alpha(g)\leq 6$.
\item[$(ii)$] $L=\PSL_5(2)$: $|g| \in \{ 2, 3, 4, 5, 7, 8, 16, 31 \}$ and $\dim \phi\geq 75$.
If $|g|\geq 5$, then $\alpha(g)=2$; if $|g|\leq 4$ then $\alpha(g)\leq 5$.
\item[$(iii)$] $L=\PSL_6(2)$:  $|g| \in \{ 2, 3, 4, 5, 7, 8, 9, 16, 31 \}$ and $\dim \phi\geq 61$.
If $|g|\geq 5$, then $\alpha(g)=2$, whereas if $|g|\leq 4$, then $\alpha(g)\leq 6$. 
\item[$(iv)$] $L=\PSL_6(3)$: $|g| \in \{ 2, 3, 4, 5, 7, 8, 9, 11, 13, 16, 121    \}$ and $\dim \phi\geq 
362$. If $|g|\geq 9$ then  $\alpha(g)=2$; if $|g|\leq 8$  then  $\alpha(g)\leq 6$.
\end{itemize}

Next, suppose that $L=\PSU_n(q)$. If $n=3$ and $q\geq 5$, the bound $\vartheta_3^-(q)\leq 3(\mu(\PSU_3(q))-1)$ 
only holds for $q=5,8,11,17$.
If $n=4$ and $q\geq 4$, the bound  $\vartheta_4^-(q)\leq 6 (\mu(\PSU_4(q))-1)$ only allows the following possibilities:
$q=4,5,7,9,11$.
If $n\geq 5$ and $(n,q)\neq (6,2)$, the bound $\vartheta_n^-(q)\leq n (\mu(\PSU_n(q))-1)$ only holds 
when $(n,q)\in \{(5,2),(5,3),(7,2)\}$.

If $L=\PSU_6(2)$, then $|g| \in \{ 2, 3, 4, 5, 7, 8, 9, 11, 16  \}$ and $\dim \phi\geq 21$.
If $|g|\geq 5$, then $\alpha(g)=2$, while if $|g|\leq 4$ then $\alpha(g)\leq 6$.
Thus, if $|g|\neq 16$, by Lemma \ref{uu3} $\phi(g)$ is not almost cyclic.
If $|g|=16$ ($g \not \in L$), we need to examine the representations $\phi$ such that $21 \leq \dim \phi\leq 30$.
Using the GAP package, we find that $\phi(g)$ is never almost cyclic.

Finally, suppose that $L=\PSp_{2n}(q)$. 
If $n=2$ and $q\geq 5$, then the bound $\sigma_2(q)\leq 5 \frac{q^{3}-q+1}{q-1}$ only holds for
$q=5,7,9,11$. If $n\geq 3$, the bound $\sigma_n(q)\leq 2n\frac{q^{n+1}-q+1}{q-1}$ only holds for
$(n,q)\in \{(3,3),(4,3)\}$.

If $L=\PSp_4(5)$, then $|g| \in \{ 2, 3, 4, 5, 8, 13\}$ and $\dim \phi\geq 40$.
If $|g|=8,13 $, then $\alpha(g)=2$, whereas if $|g|\leq 5$, then $\alpha(g)\leq 5$. In both cases Lemma \ref{uu3} 
yields a contradiction.

In view of the above, we are left to examine the groups $L$ listed in Table \ref{megaTable} below. 
But again, in each case, in view of the data shown in table, we conclude by Lemma  \ref{uu3} that $\phi(g)$ cannot be 
almost cyclic.
\end{proof}

\begin{table}[ht]
$$\begin{array}{c|c|c|c||c|c|c|c}
 L & |g|\leq & \alpha(g)\leq & \dim \phi\geq  &  L & |g|\leq & \alpha(g)\leq & \dim \phi\geq  \\\hline
\PSL_3(5) & 31 & 3  &  96 &
\PSL_3(7) & 19 & 4  & 96 \\
\PSL_3(13) & 61 & 4 & 672 & 
\PSL_4(4) & 17 &  6  & 189 \\ 
\PSL_4(5) & 31   & 6 & 248 & 
\PSL_4(7) & 25 & 6 & 1026 \\
\PSL_4(9) & 41 & 6 & 2912   & 
\PSL_4(11) & 61 & 6 & 6650 \\
\PSL_4(13) & 61 & 6 & 13176 &
\PSU_3(5)& 8 & 3 & 28\\
\PSU_3(8) & 19 & 3 & 105 & 
\PSU_3(11) & 37 & 3 & 260\\
\PSU_3(17) & 32  & 3 &  912 &
\PSU_4(4) & 17 & 6 & 220  \\
\PSU_4(5) & 13 & 6 & 272 & 
\PSU_4(7) & 43 & 6 & 1074 \\
\PSU_4(9) & 64 & 6 & 2992 & 
\PSU_4(11) & 64 &  6 & 6770 \\
\PSU_5(3) & 61 & 5  & 324  & 
\PSU_7(2) & 43 & 7  & 320 \\
\PSp_4(7) & 25 & 5 &  126 & 
\PSp_4(9) & 41 & 5 & 288 \\
\PSp_4(11) & 61 & 5 & 550 &
\PSp_6(3) & 13  & 6 & 78 \\
\PSp_8(3) & 41  & 8  & 780
  \end{array}$$
 \caption{Bounds for $|g|$, $\alpha(g)$ and $\dim \phi$ for some simple groups.}\label{megaTable}
\end{table}

\subsection{Case $(\mathrm{B})$}
 
We now consider case (B): $g\notin A_d$ and all the constituents of the restriction of $\phi$ to $L$  are Weil 
representations  of $L$. We start with recording some arithmetical inequalities to be used when applying Lemma \ref{uu3}.

\begin{lemma}\label{no2}
Let $r$ be a prime, $q= r^a$.
\begin{itemize}
\item[(1)] Let $n\geq 5$. Then $\frac{q^n-1}{q-1}-2> n(q-1)$. Furthermore, $\frac{q^3-1}{q-1} -2 > 4(q-1)$ and 
$\frac{q^4-1}{q-1} -2 > 7(q-1)$.
\item[(2)] Let $n\geq 3$ and $q$ odd. Then $\frac{q^n-1}{2}> 2n(q-1)$. Furthermore, $\frac{q^2-1}{2}> 6(q-1)$ unless 
$3\leq q \leq 11$.
\item[(3)] Let $n\geq 5$ be odd. Then $\frac{q^n-q}{q+1}> n(q-1)$. Furthermore, $\frac{q^3-q}{q+1}> 4(q-1)$, unless 
$q=2,3,4$.
\item[(4)] Let $n\geq 6$ be even. Then $\frac{q^n-1}{q+1}> n(q-1)$. Furthermore, $\frac{q^4-1}{q+1}> 6(q-1)$, unless 
$q=2$.
\end{itemize}
\end{lemma}

\begin{proof}
Elementary straightforward computations. 
\end{proof}

\begin{lemma}\label{3hh}
Let $L$ be one of the simple groups $\PSL_n(q)$, $n>2$; $\PSU_n(q)$, $n >2$; $\PSp_{2n}(q)$, $n > 1$, $q$ odd. 
Let $G=\lan g,L\ran \subseteq \Aut L$, where $g$ is a $p$-element. 
Let $\phi$ be a cross-characteristic irreducible projective  $\F$-representation of $G$ non-trivial on 
$L$, and suppose 
that $p$ is coprime to $|L|$. Then $\phi(g)$ is not almost cyclic.
\end{lemma}

\begin{proof}
Observe that if $p$ is coprime to $|L|$, then $p>3$, and hence $ |g|$ divides $|\Phi|$. This implies   $3<|g|<q- 1$. 
Now, by Lemma \ref{md6}, the minimum degree of a non-trivial irreducible projective representation of $L$ is 
not less than the number at the left hand side of the items (1) to (4) of Lemma \ref{no2}, unless 
$(n,q)\in\{(3,2),(3,4),(4,2),(4,3)\}$ in case (1), and $(n,q)\in\{(4,2),(4,3)\}$ in case (4). However, all these 
exceptions occur for $q=r$, that is $|\Phi|=1$, or $q=4$, and therefore are irrelevant here.
Similarly, all the exceptions listed in  cases (2), (3) and (4)  of  Lemma \ref{no2} are irrelevant here.
Furthermore, by Lemmas \ref{GS1} and \ref{GS2}, $\al(g)\leq n$ if $L$ is not symplectic, and $\al(g)\leq 2n$ in the 
symplectic case, with some exceptions described there. Again, as $2<|g|
<q- 1$, all those exceptions are ruled out. By Lemma \ref{no2} and Lemma \ref{uu3}, it follows that $\phi(g)$ is not 
almost cyclic. 
\end{proof}

We now suppose that $p$ is a prime dividing $|L|$ and distinguish two cases depending on the divisibility of $q$ by $p$.

 \subsubsection{Case $(p,q)=1$}

\begin{lemma}\label{p67b}
Let $L=\PSL_n(q)$, with $n>2$, and  let $p>2$ be such that $p$ divides $|L|$ and $(p,q)=1$.
Let $1 \neq g\in \Aut L$ be a $p$-element and let $G=\langle g,L\rangle$.
Suppose that 
$g\notin A_d$. Then $n|g|< \frac{q^n-1}{q-1}-2$, and $g$ is not almost cyclic in any cross-characteristic projective 
irreducible representation of $G$ that is non-trivial on $L$.
\end{lemma}

\begin{proof}
Let $H = \GL_n(q)$. As $p>2$, $g\in (H\cdot \Phi)/Z(H)\subseteq \Aut L$. By Lemmas \ref{ge7} and \ref{za4}, $|g|\leq 
\eta_p(H)=p^t |q^e-1|_p\leq \frac{n}{e}(q^e-1)$, where $p^t e\leq n < p^{t+1} e$.
Observe that, since $p$ is odd, the assumption $g \notin A_d$ implies that $|\Phi|_p =p^m >1$. This in turn implies  
$q\geq 8$, and hence $\frac{q^n-1}{q-1}-2$ is the minimum degree of a non-trivial cross-characteristic projective 
irreducible representation of $L$ (see Lemma \ref{md6}).
Notice that if $e< \frac{n}{2}$, then $\frac{n^2}{e}  (q^e-1)< \frac{q^n-1}{q-1}-2$. Therefore the statement is proven 
for $n> 2e$.

Next, suppose that $n\leq 2e$, and observe that the assumption $n>2$ implies $e>1$.
Since $p^te \leq n\leq 2e$ and $p>2$, it follows that $p^t=1$.
Let $q_1$ be such that $q_1^{p^m}=q^e$. By Lemma \ref{f4t}, we get
$|q^e-1|_p= p^m |q_1-1|_p$, whence $|g|\leq \eta_p(H)\leq p^m(q_1-1)$.
We need to show that $n p^m (q_1-1) < \frac{q^n-1}{q-1}-2$, and for this it is enough to show that $n p^m (q_1-1) <  
q^{n-1}+q^{n-2}-2$.
Since $e\leq n$ and $e < p$ (note that $p$ is a Zsigmondy prime for the pair $(q,e)$),
it suffices to show that $2 p^{m+1} (q_1-1) < q^{e-1}+q^{e-2}-2=q_1^{p^m-1}+q_1^{p^m-2}-2$.
This is true  except when $m=1$ and
either $p=3$ or
$(p,q_1)\in \{(5,2),(7,2)\}$.
If $p=5,7$ and $q_1=2$, then $q^e=2^5,2^7$, whence $q=2$ and $e=p$, a contradiction, as $e<p$.
If $p=3$, then $2\leq e <p $ yields $e=2$. Thus, we are left to consider the cases $q^2=q_1^3$ and  $n=3,4$.
Notice that $12(q_1-1)< q_1^3+5$. It follows that, for $n=3$, the inequality $n p^m (q_1-1) < \frac{q^n-1}{q-1}-2$ 
reduces to
$9(q_1-1)< q^2+q-1$, and in fact $9(q_1-1)<12(q_1-1)< q_1^3+5\leq q^2+q-1$, as $q\geq 8$. For $n=4$, the same 
inequality 
reduces to $12 (q_1-1 ) < q^3+q^2+q-1$. Which is true, since $12(q_1-1)< q_1^3+5 <q^3+q^2+q-1$. This ends the proof.
\end{proof}

\begin{lemma}\label{newp2p}
Let $L=\PSL_n(q)$ with $n>2$, $q$ odd.
Let $g\in \Aut L$ be a $2$-element such that $g^2\neq 1$, and let $G=\langle g,L\rangle$.
Then  $g$ is not almost cyclic in any cross-characteristic projective irreducible representation of $G$ that is 
non-trivial on $L$.
\end{lemma}

\begin{proof}
Let $H = \GL_n(q)$, and let  $\Gamma$ be the group generated by the inverse-transpose automorphism of $H$. Clearly 
$|\Gamma|=2$.
Then $(H\cdot \Phi \cdot \Gamma)/Z(H)\cong \Aut L$, and therefore it suffices to prove the statement for $H\cdot 
\Phi\cdot \Gamma$ in place of $\Aut L$.
Obviously, $\eta_2(H\cdot \Phi\cdot \Gamma)\leq 2\eta_2(H\cdot \Phi)$.
By Lemma \ref{za4bis},  $\eta_2(H\cdot \Phi)=\eta_2(H)$.
Let $m\geq 0$ such that $q=q_0^{2^m}$, where $q_0$ is not a square.
By Corollary \ref{ce7} and Lemma \ref{nn9} we get $\eta_2(H)=2^m \eta_2(\GL_n(q_0))\leq 2^{m+t} (q_0+1)$,
where $t$ is such that $2^t \leq n< 2^{t+1}$ (notice that $t\geq 1$ as $n>2$).
Hence, we have $\eta_2(H\cdot \Phi\cdot \Gamma)\leq 2^{m+t+1}(q_0+1)$.

We now show that if $n\geq 5$ and $(n,q)\neq (5,3)$, then $n\cdot 2^{m+t+1}(q_0+1) < \frac{q^n-1}{q-1}-2$. Under these 
restrictions, the statement of the lemma follows, by
Proposition \ref{GS1} and  Lemmas \ref{uu3} and \ref{md6}.
Clearly, $n \cdot 2^{m+t+1}\leq 2^{m+2t+2}$ and $\frac{q^n-1}{q-1}-2\geq q^{n-1}+q^{n-2}-1\geq q^{2^t-1}+q^{2^t-2}-1$, 
so it suffices to prove that
$2^{m+2t+2}(q_0+1)< q_0^{2^{m+t}-2^m}+q_0^{2^{m+t}-2^{m+1}}-1$.
Direct computation shows that this is true except when $(m,t)=(0,2)$ and
$q_0\leq 7$.
So, suppose that $m=0$ and $t=2$: we have to prove that $8n(q_0+1)<\frac{q_0^n-1}{q_0-1}-2$ for $n=5,6,7$.
This is obviously true except when $(n,q_0)=(5,3)$.

Now, suppose that $n=3$. In this case the statement will be proven if we show that
$2^{m+4}(q_0+1) < q_0^{2^{m+1}}+q_0^{2^m}-1$.
Computation shows that this is true except when either
$m=0$ and $q_0\leq 13$ or $(m,q_0)=(1,3)$.

Next, suppose that $n=4$.  For $q\neq 3$, the statement will be proven once we show that
\begin{equation}\label{eqq1}
6\cdot 2^{m+3}(q_0+1) < q_0^{3\cdot 2^m}+q_0^{2^{m+1}}+q_0^{2^m}-1.
\end{equation}
Now, if $m\geq 3$, then  $6\cdot 2^{m+3} < 3^{3\cdot 2^m-1}$ and
$6\cdot 2^{m+3} < 3^{2^{m+1}}$.
Direct computation for $m\leq 2$ shows that \eqref{eqq1} holds except for $m=0$ and $q_0=3,5$.

We finally deal with the exceptions arisen above. Let $\phi$ be a cross-characteristic projective irreducible 
representation of $G$ non-trivial on $L$.
In the cases listed in the following table we easily get a contradiction by Lemma  \ref{uu3}:

\begin{center}
\begin{tabular}{c|c|c||c|c|c}
$L$ &  $\eta_2(L)$ & $\dim \phi\geq $ & $L$ &  $\eta_2(L)$ & $\dim \phi\geq $ \\\hline
$\PSL_3(5)$ & $8$ & $29$ & 
$\PSL_3(9)$ & $16$ & $89$ \\
$\PSL_3(11)$ & $8$ & $131$ & 
$\PSL_3(13)$ & $8$ & $181$ \\
$\PSL_4(5)$  & $8$ & $154$ & 
$\PSL_5(3)$ & $16$ & $119$ 
\end{tabular}
\end{center}

Thus, we are left with the cases where $L = \PSL_3(3)$, $\PSL_3(7)$ and $\PSL_4(3)$. 

If $L=\PSL_3(3)$, then $\eta_2(L)=8$ and $\dim \phi\geq 11$. Using GAP we see that if $|g|=4$, then $\alpha(g)=2$. By 
Lemma \ref{uu3}, we conclude that $\phi(g)$ is not almost cyclic.
If $|g|=8$, we are left to consider the case $11\leq \dim \phi\leq 14$. Again using GAP, we see that $\phi(g)$ is not 
almost cyclic.

If $L=\PSL_3(7)$, then $\eta_2(L)=16$ and $\dim \phi\geq 55$. Using GAP, we see that for any $2$-element $g$ such 
that $g^2\neq1$ we have $\alpha(g)=2$, and hence again, by Lemma \ref{uu3}$, \phi(g)$ is not almost cyclic.

Finally, if $L=\PSL_4(3)$, then $\eta_2(L)=8$ and $\dim \phi\geq 26$. 
If $|g|=8$, then $\alpha(g)=2$; if $|g|=4$, then $\alpha(g)\leq3$. In both cases, by Lemma \ref{uu3} $\phi(g)$ is not 
almost cyclic.
This ends the proof.
\end{proof}
 
Let $L=\PSU_n(q)$, $n>2$, and let $p$ be a prime such that $(p,q)=1$ and $p$ divides $|L|$. Let $\Phi=\Gal(\F_q/\F_r)$, 
and set $p^m=|\Phi|_p$. Thus $p^m< r^{p^m}\leq q$. Let $g\in \Aut L$ be a $p$-element and set $G=\lan L,g\ran$.
Recall that, by Propositions \ref{GS1} and \ref{GS2}, 
$\lan L,g\ran$ is generated by $n$ conjugates of $g$.   Finally, set $\Phi_2=\Gal(\F_{q^2}/\F_r)$, so that $\Aut 
L=A_d\cdot \Phi_2$.

\begin{lemma}\label{e24}
Let $p>2,m>0$ and  $e\neq 2\pmod 4$.
Let $1 \neq g\in \Aut L$ be a $p$-element. 
Set $G=\lan L,g\ran$. 
Then $n|g|<(q^n-q)/(q+1)$ if $n$ is odd,  $n|g|<(q^n-1)/(q+1)$ if $n$ is even, and
$\phi(g)$ is not almost cyclic in any cross-characteristic irreducible projective representation $ \phi$ of $G$ 
non-trivial on $L$. 
\end{lemma}

\begin{proof} 
Let $H=\GU_n(q)$. As $e= e_p(q) \neq 2$, $p$ is coprime to $q+1=|Z(H)|$. It follows that  $|A_d/L|$ is coprime 
to $p$. So 
$(H\cdot \Phi_2 )/Z(H)\subseteq \Aut L$ contains a Sylow $p$-subgroup of $\Aut L$. 

It is known (see for instance \cite[Lemma 3.5]{HSZ}), that $H$ contains a subgroup isomorphic to $H_1:=
\GL_{\lfloor n/2 \rfloor}(q^2)$, 
where $\lfloor n/2 \rfloor$ is the integral part of $n/2$, and 
$|H:\GL_{\lfloor n/2 \rfloor}(q)|$ is   coprime to $p$. Moreover, $H_1$ can be chosen 
to be invariant under $\Phi_2$, and therefore $H_1\cdot \Phi_2$ contains a Sylow $p$-subgroup 
of $H\cdot\Phi_2$. So, by Lemma \ref{za4},  $|g|\leq \eta_p(H_1)$. Set $l=\lfloor n/2 \rfloor$ (that is, $l=n/2$ if $n$ 
is even, $l=(n-1)/2$ if $n$ is odd). As $p>2$, we have $|\Phi|_p=|\Phi_2|_p>1$; in particular, $q \geq r^p\geq 8$. 

Note that $|q^{2e}-1|_p=|q^e-1|_p$ if $e$ is odd, and $e_p(q^2)=e_p(q)/2$ if $e$ is even.  So in both cases
$|(q^2)^{e_p(q^2)}-1|_p=|q^e-1|_p$. 
Therefore, by Lemma \ref{ge7},  $|g|\leq \lfloor n/2 \rfloor\cdot |q^{e}-1|_p$, whence $|g|\leq \lfloor n/2 \rfloor 
\cdot |q^{e}-1|_p\leq l(q^l-1)$   (since  $e\leq l$).  

Suppose first that $n$ is even. In order to prove the statement, by Propositions \ref{GS1} and \ref{GS2}, Lemma \ref 
{uu3} and Lemma \ref{md6}, it will be enough to show that $n|g|<(q^n-1)/(q+1)$. For this purpose, it suffices to show 
that $(2l)\cdot l \cdot (q+1) < q^l + 1$.
This is true if $l>2$, unless $l=3,q\leq 4$; $l=4, q \leq 3$; $5 \leq l \leq8$, $q = 2$. However, all these exceptions 
are ruled out, since $q\geq r^p$.
Next, suppose that $l= 2$, so that $n=4$ and $e=1$, that is, $p\mid (q-1)$. 
Then $8(q+1) < q^2+1$, unless $q \leq 8$. As $q \geq r^p$, this forces $q = 8$. However, as $p\mid (q-1)$, this in turn 
implies $p = 7$, a contradiction.

Now, suppose that $n$ is odd. In this case, it will be enough to show that $n|g|<(q^n-q)/(q+1)$. For this purpose, it 
suffices to show that $(2l+1)\cdot l \cdot (q+1) < q^l + 1$ (since $q^{n-1} < q^n - q$). This is true if $l>2$, unless 
$l=3,q\leq 5$; $l=4, q \leq 3$; $5 \leq l \leq 9$, $q = 2$. However, all these exceptions are ruled out, since $q\geq 
r^p$. If $l= 2$, that is, $n=5$, 
we have $10(q^2-1) < (q^5-q)/(q+1)$, unless $q \leq 3$. If $l=1$, that is, $n=3$, $|g| \leq (q-1)$ and the inequality 
$3(q-1) < q(q-1)$ holds unless $q\leq 3$. As $q\geq 8$, we are done.
\end{proof}

\begin{lemma}\label{5.7}
Let $L=\PSU_n(q)$, with $n>2$.
Let $1\neq g \in \Aut L$ be a $p$-element, for some prime $p>2$ such that $e_p(q) \equiv 2 \pmod 4$, and let
$G=\langle g,L\rangle$.
Then $g$ is not almost cyclic in any cross-characteristic projective irreducible representation  of $G$
that is non-trivial on $L$, unless $n>3$ and $|\Phi|_p = 1$ (and hence $g\in A_d$).
\end{lemma}

\begin{proof}
First of all, we observe that, by Lemmas \ref{ge7} and \ref{za4},
$|g|\leq \eta_p(\GL_n(q^2)\cdot \Phi_2)=\eta_p(\GL_n(q^2) )=p^t|(q^2)^{e'}-1 |_p$, where
$e'=e_p(q^2)$ and $p^te'\leq n< p^{t+1}e'$.
By the definition of $e$, $e=2e'$ and, since $e$ is even, we obtain that 
$|g|\leq p^t|q^e-1|_p =p^t|q^{e/2}+1|_p$.

Assume first that $n=3$. 
Then $e_p(q)=2$, that is, $p$ divides $q+1$. 
By Lemma  \ref{md6}(3), $\dim \phi\geq q^2-q$. Thus, the statement will be proven once we show that $3|g|\leq q^2-q$ 
(see Proposition \ref{GS2}(3)).
Since $p\mid (q+1)$, we have $e=2$, whence $|g|\leq p^t|q^2-1|_p=p^t|q+1|_p\leq p^t(q+1)$.
If $p>3$, then $q \geq 9$; moreover, $p^t\leq 3 < p^{t+1}$ implies $p^t=1$. It follows that $3|g|\leq 3(q+1)\leq 
q^2-q$, 
and we are done.
If $p=3$, then  $p^t\leq 3 < p^{t+1}$ yields $p^t=3$ and $|g|\leq 3(q+1)$.
Now, $3|g|\leq 9(q+1)\leq q^2-q$ if $q \geq 11$; thus we are only left to examine the cases $q=5,8$. In both cases, we 
get a contradiction with Lemma \ref {uu3}.

Next, assume that $n>3$. In this case, we will show that
$n|g|\leq  \frac{q^n-(-1)^n}{q+1}-1$.

Let us write $q=q_0^{p^m}$, where $p^m = |\Phi |_p$. Then, by Lemma \ref{f4t}, 
$|g|\leq p^{t+m}|q_0^{e/2}+1|_p\leq p^{t+m}(q_0^{e/2}+1)\leq n
p^m(q_0^{e/2}+1)\leq n p^m(q_0^{\frac{p-1}{2}} +1)$ (since $e<p$).

If $m\geq 1$ and   $(n,q)\neq(4,8),(5,8)$, then $n^2p^m(q_0^{\frac{p-1}{2}}+1)+1 \leq q_0^{\lfloor
\frac{n}{2}\rfloor p^m}$.
Hence $(n|g|+1)(q+1)\leq q_0^{\lfloor \frac{n}{2}\rfloor p^m} (q_0^{p^m}+1)\leq q_0^{np^m}-2$, whence $n|g|\leq  
\frac{q^n-(-1)^n}{q+1}-1$.
If $L\in \{\PSU_4(8),\PSU_5(8)\}$, then $p=3$, $e=2$ and $t=1$, whence $|g|\leq 3\cdot |8+1|_3=27$.
Since $4\cdot 27\leq 454=\frac{8^4-1}{8+1}-1$ and
$5\cdot 27\leq 3640=\frac{8^5+1}{8+1}-1$, we are done.
\end{proof}

 \begin{lemma}\label{nn5u} 
Let $L=\PSU_n(q)$, with $n>2$ and $q$ odd.
Let $g \in \Aut L$, $g\notin A_d$, be a $2$-element such that $g^2\neq 1$, and set $G=\langle g,L\rangle$.
Suppose $(n,q)\not \in \{(3,3),(4,3)\}$.
Then $g$ is not almost cyclic in any cross-characteristic projective irreducible representation  of $G$
that is non-trivial on $L$.
\end{lemma}

\begin{proof}
Set $H=\GU_n(q)$, $H_1=\GL_n(q^2)$ and $q=q_0^{2^m}$, where $2^m=|\Phi|_2$.
We claim that $\alpha(g)\eta_2(G) < \dim \phi$ for every cross-characteristic projective irreducible representation 
$\phi$ of $G$ non-trivial on $L$, whence the statement by Lemma \ref{uu3}.
Clearly, we may replace $G$ with $H\cdot \Phi_2$.
By Lemmas \ref{za4bis} and \ref{nn9}, $\eta_2(H_1\cdot \Phi_2)=\eta_2(H_1)=2^t|q^2-1|_2$, where $2^t\leq n< 2^{t+1}$. 
Since $H\leq H_1$, it follows that  $\eta_2(H)\leq 2^t|q^2-1|_2\leq 2n(q+1)$.

Suppose that  $n\geq 5$. Then, by Proposition \ref{GS1},  $\alpha(g)\leq n$ and by Lemma \ref{md6} $\dim \phi\geq 
\frac{q^n-q}{q+1}$. Now, if either $n\geq 7$ or $n=5,6$ and $q\geq 5$, we have $2n^2(q+1)< \frac{q^n-q}{q+1}$, and we 
are done.
If $(n,q)=(6,3)$, then we find that $\eta_2(H_1)=16$ and $\dim \phi \geq 182$, and we are done.
If $(n,q)=(5,3)$, then, again using the GAP package, we find that $\eta_2(G)=16$ and moreover, $\alpha(g)\leq 5$ if 
$|g|=4$, $\alpha(g)\leq 3$ if $|g| = 8$ and  $\alpha(g)=2$ if $|g|=16$. Since $\dim \phi\geq60$, we are done.

Next, suppose  that $n=4$. Then $\dim \phi\geq q^3-q^2+q-1$ and, by Proposition \ref{GS2}, $\alpha(g)\leq 6$. Since 
$6\cdot 8(q+1)  < q^3-q^2+q-1 $ if $q\geq 9$, we are left to consider the cases $q=5,7$.
For $q=5$ we find that $\eta_2(G)=8$, whence the claim, since $\dim \phi \geq 104$.
For $q=7$ we find that $\eta_2(G)\leq 32$ and $\dim \phi \geq 300$, whence the claim.

Finally, suppose that $n=3$. Then $\dim \phi \geq q^2-q$, $\alpha(g)\leq 3$ and $\eta_2(H_1)\leq 4(q+1)$. Since
$3\cdot 4(q+1) < q^2-q$ for $q\geq 17$, we are left to consider the cases $q=5,7,9,11,13$.
For $q=5$, we find that $\eta_2(G)=8$ and $\dim \phi\geq 20$. Moreover, if $|g|=4$,
then $\alpha(g)\leq 3$;
if $|g|=8$, then $\alpha(g)=2$. Thus in both cases $\phi(g)$ is not almost cyclic.

For $q=7$, we find that $\eta_2(G)=16$ and $\dim \phi\geq 42$.
If $|g|=4$, or $8$, then $\alpha(g)\leq 3$;
if $|g|=16$, then $\alpha(g)=2$. Thus, we are done with the case $q=7$.
Next, for $q=9$, $\eta_2(G)=16$ and $\dim \phi\geq 72$; 
for $q=11$,  $\eta_2(H_1)=16$ and $\dim \phi \geq 110$; for $q=13$, $\eta_2(H_1)=2\cdot 8=16$ and $\dim \phi \geq 156$. 
It follows that also in these cases, we are done by Lemma \ref{uu3}.
This ends the proof.
\end{proof}

Let $L=\PSp_{2n}(q)$ and $H=\Sp_{2n}(q)$, where $q$ is odd and $n>1$.  
Let $p^m=|\F_q:\F_r|_p$. Recall that $\mathrm{CSp}_{2n}(q)$ denotes the conformal symplectic group. It is well  known 
that
$A_d=\mathrm{CSp}_{2n}(q)/Z$,
 where $Z$  is the group of scalar matrices in $\GL_{2n}(q)$. 
Furthermore, $|A_d:L|=2$, as $q$ is odd.

\begin{lemma}\label{2ss}
Let  $L$ be as above, with $p$ odd. Let $1\neq g\in \Aut L$ be a $p$-element and set 
 $G=\lan g, L\ran$. Let $\phi$ be a cross-characteristic irreducible projective representation of $G$ such that 
$\dim\phi>1$.  
Suppose that  $g\notin L $.  Then $\phi(g)$ is not almost cyclic.
 \end{lemma}

\begin{proof}
 By Lemma \ref{za4}, $\eta_p(\Aut L)\leq \eta_p(H)$. As $g\notin L$, the assumption $p>2$ implies $q\geq 8$, and hence, 
as $q$ is odd, $q\geq 27$. By Lemma \ref{uu3}, it suffices to show that $2n|g|<\frac{q^n-1}{2}\leq \dim\phi$.

Suppose first that $e$ is odd. 
  
Then $H$ contains a subgroup $H_1\cong \GL_n(q)$, and  $|H|_p=|H_1|_p$ (again, see \cite{W55}).
Thus, by the above, $|g|\leq \eta_p(H_1)$. The latter, by Lemma \ref{ge7}, equals
$p^t|q^e-1|_p$, where $p^te\leq n <p^{t+1}e$. If $n>2$, it follows from Lemma \ref{p67b} that 
$n|g|<\frac{q^n-1}{q-1}-2$. 
Since $2(\frac{q^n-1}{q-1}-2)<\frac{q^n-1}{2}$ for $q\geq 7$, we are done.
So, let $n=2$.  Then $e=1$ and 
 $|g|\leq \eta_p(H_1)= |q-1|_p$.  Thus $4|g|\leq 4|q-1|_p$, and 
 one easily checks that $4|g|< (q^2-1)/2$ for $q>7$. 

Next, let $e$ be even. 

Then it is known that $\Sp_{2n}(q)$ contains a Sylow $p$-subgroup of $\GL_{2n}(q)$ (again, see \cite{W55}).
Write $2n=ae+b$, where $0\leq b < e$ and, as usual, let $q=q_0^{p^m}$, where $p^m=|\Phi|_p$. 
It is easy to observe that $|\GL_{2n}(q)|_p=|\GL_{ae}(q)|_p$.
Therefore $|\GL_{ae}(q)|_p=|\Sp_{ae}(q)|_p$, whence $|\Sp_{2n}(q)|_p=|\Sp_{ae}(q)|_p.$ Thus $|g|\leq 
\eta_p(\GL_{ae}(q))$. 

Suppose first that $2n=ae$. Observe that, by Lemma \ref{f4t}(1), $e=e_p(q) =e_p(q_0^{p^m})=e_p(q_0)$, and hence  
$p$ divides $q_0^e-1$. It then follows, by Lemmas
\ref{f4t}(2) and \ref{ge7}, that $|g|\leq \eta_p(\GL_{ae}(q))=p^{t+m}|q_0^e-1|_p$, where $p^t\leq a < p^{t+1}$.
We wish to show that $ae|g|\leq 2n p^{t+m}|q_0^e-1|_p \leq \frac{4n^2}{e} p^m (q_0^e-1)<\frac{q^{ae/2}-1}{2}$.
To this purpose, it suffices to show that 
$4n^2 p^m(q_0^{2n}-1)<   \frac{q_0^{n p^m}-1}{2}$.
Now, $8n^2 p^m q_0^{2n}< q_0^{np^m}$ holds for any $p\geq 3$, any $n\geq 2$ and any $q_0\geq 3$, unless $p=3$, $m=1$ 
and one of the following occurs:
$$\begin{array}{clccl}
(i)  & q_0=3 \equad  n\leq 6; & \qquad  & (ii) & q_0=5 \equad n\leq 3\\
(iii) & q_0=7 \equad n=2;& \qquad & (iv)  & q_0=9 \equad n=2.
\end{array}$$
As $p\nmid q$, cases $(i)$ and $(iv)$ are ruled out. On the other hand, in all the other cases, we still have 
$2n|g|\leq 
\frac{q^n-1}{2}$. Namely, in case $(ii)$, for $n=2$ we have $2n|g|=4|g|\leq 4\eta_3(\GL_4(5^3))=36< 
\frac{125^{2}-1}{2}$, and  for $n=3$ we have $6|g|\leq 6\eta_3(\GL_6(5^3))=162< \frac{125^3-1}{2}$.
In case $(iii)$ we have $2n|g|=4|g|\leq 4\eta_3(\GL_4(7^3))=108< \frac{343^2-1}{2}$.

Finally, suppose  $2n> ae$. Then $2n|g|=\frac{2n}{ae} ae|g|\leq \frac{2n}{ae}\cdot \frac{(q^{ae/2}-1)}{2}< 
\frac{q^n-1}{2}$ for every $n\geq 2$.
\end{proof}

\begin{lemma}\label{32s}
Let $G=\Sp_{2n}(q)$ and $H=\mathrm{CSp}_{2n}(q)$, where $q$ is odd.
\begin{itemize}
\item[(1)] $G$ has two inequivalent irreducible $\F$-representations of degree $(q^n-1)/2;$ they do not extend to $H$.
\item[(2)] Let $\ch \F\neq 2$. Then $G$ has two inequivalent irreducible $\F$-representa\-tions of degree 
$(q^n+1)/2$. They do not extend to $H$.
\item[(3)] If $\phi$ is an irreducible  $\F$-representation of $H$ with $\dim\phi>1$,  then $\dim\phi\geq q^n-1$.
\end{itemize}
\end{lemma}

\begin{proof}
(1) and (2):
 It is well known that, over the complex numbers, $G$ has two inequivalent irreducible representations  
of degree $(q^n-1)/2$ and two  inequivalent irreducible representations  
of degree $(q^n+1)/2$ (see \cite{Ge}). 
 Moreover, it follows from \cite[Proposition 1.4]{Ge} that the representations of equal degrees can be obtained from 
each other by 
twisting by an involution $h\in H$, and hence none of them extends to $H$. The reductions modulo an odd prime distinct 
from $\ell$   of these representations 
 remain irreducible and inequivalent (see \cite{War}). As for the representations of degree $(q^n-1)/2$ the 
reduction 
modulo 2 remains 
irreducible,   whereas the reduction modulo 2 of the representations of degree $(q^n+1)/2$ has a composition factor of 
degree 
$(q^n-1)/2$. This implies that the irreducible $\F$-representations of $G$   do not extend to $H$. Moreover, it is 
clear 
that $H$ has an 
irreducible $\F$-representation of degree $q^n-1$, whose restriction to $G$ splits into two inequivalent 
irreducible 
representations of $G$ of degree 
$(q^n-1)/2$. \\
  \noindent (3):  By \cite[Theorem 2.1]{GMST},  it is known that, if $\tau$ is an  irreducible   $\F$-representation of 
$G$ of degree 
greater that $(q^n+1)/2$, then $\dim\tau\geq 
  (q^n-1)(q^n-q)/2(q+1)$ . It follows from Clifford's theorem that 
the same is true for $\phi$, provided $\dim\phi>q^n+1$.  

Now, suppose $\dim\phi\leq q^n+1$. Again by \cite[Theorem 2.1]{GMST}, it follows from Clifford's theorem that
the composition factors of $\phi|_G$ are of degree $(q^n\pm 1)/2$. On the other hand, by the above, $\dim\phi=q^n\pm 
1$. 
Taking into account that, if $\ch \F=2$, there are no irreducible representation of $G$ of degree $(q^n+1)/2$, item (3) 
follows. 
 \end{proof}

\begin{lemma}\label{4ss}
Let $L=\PSp_{2n}(q)$, where $n\geq 2$ and $q$ is odd.
Let $1\neq g \in \Aut L$ be a $2$-element such that $g^2\neq 1$, and let $G=\langle g,L\rangle$.
Suppose that $(2n,q) \neq (4,3)$.
Then $g$ is not almost cyclic in any cross-characteristic projective irreducible representation  $\phi$ of $G$ that is 
non-trivial on $L$.
\end{lemma}

\begin{proof}
Let $G_1=\mathrm{CSp}_{2n}(q)\cdot \Phi$. Then $G_1\subseteq \GL_{2n}(q)\cdot \Phi$.
By Lemmas \ref{za4bis} and \ref{nn9}, we have $\eta_2(\Aut L)\leq \eta_2(G_1)\leq \eta_2(\GL_{2n}(q)\cdot \Phi)
=\eta_2(\GL_{2n}(q)) \leq 2^{t+1}(q+1)$,
where $2^t \leq n < 2^{t+1}$.
For $n>2$, we get that  $\phi(g)$ is not almost cyclic by Lemma \ref{uu3} once we show that $2n|g|< \frac{q^n-1}{2}$.
Now, $2n|g|\leq 2n\cdot 2^{t+1}(q+1)\leq 4n^2(q+1)$ and $4n^2(q+1)< \frac{q^n-1}{2}$ for all $n\geq 3$ and $q$ odd, 
unless $(n,q)\in \{(3,3),(3,5),(3,7), (4,3),(4,5), (5,3), (6,3) \}$.
For $L$ as listed in the Table \ref{megaTable}, by Lemma \ref{uu3}, we conclude that $\phi(g)$ is not almost cyclic.

If $L=\PSp_6(3)$, then $|g|\leq 8$ and $\dim \phi\geq 13$. If $|g|=4$, then $\alpha(g)\leq 3$,  and hence, by Lemma 
\ref{uu3}, $\phi(g)$ is not almost cyclic. 
If $|g|=8$, then $\alpha(g)=2$; hence we only need to examine the representations $\phi$ such that
$13\leq \dim\phi\leq 14$. Using the
GAP package we see that $\phi(g)$ is not almost cyclic.

If $L=\PSp_{12}(3)$, we have $|g|\leq \eta_2(\GL_{12}(3))=32$, $\dim \phi \geq 364$ and $\alpha(g)\leq 12$.
Thus, if $|g|=4,8,16$, we are done by Lemma  \ref{uu3}.
If $|g|=32$, then $g \not \in L$ (since $\eta_2(\PSp_{12}(3))=16$). So, we can refine the bound using Lemma \ref{32s}, 
and again we are done by Lemma \ref{uu3}.

Suppose now that $n=2$ and $q\geq 5$. Then it is enough to show that $5|g|\leq \frac{q^2-1}{2}$ (see Proposition 
\ref{GS2}(6)).
By the above $5|g|\leq 5\cdot 2^2(q+1)$, and $20(q+1)< \frac{q^2-1}{2}$ for all $q> 41$.
Assume $q \leq 41$. Then we may refine the previous bound, obtaining that 
$5|g|\leq 5\eta_2(\GL_{4}(q))\leq \frac{q^2-1}{2}$
unless $q=5,7,9,11,17,31$. In these exceptional  cases we get the following:
$$\begin{array}{c|c|c|c||c|c|c|c}
L &   |g| & \alpha(g) &   \dim \phi\geq  &  L &   |g| & \alpha(g) &   \dim \phi\geq    \\\hline
\PSp_4(5) &  =4 & \leq 3 & 12  & \PSp_4(9)  &  =16 & =2 & 40 \\
\PSp_4(5) &  =8 & =2 & 12  &  \PSp_4(11)  & \leq 8 & \leq 3 & 60 \\
\PSp_4(7) &   \leq 8 & \leq 3 & 24 & \PSp_4(17)  & \leq 32 & \leq 3 & 144  \\ 
 \PSp_4(7) & =16 & =2 & 24  & \PSp_4(31)  & \leq 64 & \leq 5 & 480 \\
\PSp_4(9)  & \leq 8 & \leq 3 & 40     
  \end{array}$$

Hence, by Lemma \ref{uu3}, we are only left to consider the following instances:\\
\noindent $(i)$ $L=\PSp_4(5)$ and $|g|=8$. Then $g\not \in L$, whence $\dim \phi\geq 24$. Thus, by Lemma 
\ref{uu3}, 
$\phi(g)$ is not almost cyclic.\\
\noindent $(ii)$ $L=\PSp_4(7)$ and $|g|=16$.
In this case $g \not \in L$, whence $\dim \phi\geq 48$ (again by Lemma \ref{32s}), which   implies, by Lemma \ref{uu3}, 
that $\phi(g)$ is not almost cyclic.
\end{proof}

\subsubsection{Case $p\mid q$}

In this subsection we examine the case where $p$ divides $q$. 
Although, for $p$ odd, we might refer to the earlier paper \cite{DMZ10}, the approach chosen in the present paper (for 
all $p$) 
appears to be more efficient than in \cite{DMZ10} (and may lead to a shorter proof of the results in  \cite{DMZ10}).

First, let $H = \GL_n(q)$ and let $\Phi$ be the group of field automorphisms of $H$.
Suppose that $t$ and $m$ are such that $p^t < n \leq p^{t+1}$ and $|\Phi|_p=p^m$.
Write $q=q_0^{p^m}$. Observe that $\eta_p(H) =  p^{t+1}$. Indeed, $p^{t+1}$ is exactly the order of a unipotent element 
of $\GL_n(q)$ consisting of single Jordan block of size $n$. Furthermore, note that $\eta_p(H) =  p^{t+1}$ also for 
$H = \SL_n(q), \GU_n(q)$ and $\SU_n(q)$.

Next, we need the following arithmetical lemma.

\begin{lemma}\label{LU}
Let $p$ be a prime and let $q$ be a $p$-power.
Let  $n$ be an integer greater than $4$, and suppose that $t,m$ and
$q_0$ are defined as above.
 If $p=2$, then
$n2^{t+m+2}\leq \frac{q^n-q}{q+1}$
unless  $q=2$ and $5\leq n \leq 9$.
If $p$ is odd, then $np^{t+m+1}\leq \frac{q^n-q}{q+1}$.
Furthermore, $n2^{t+m+2}\leq \frac{q^n-1}{q-1}-2$, unless  $q=2$ and $n=5,6$.
\end{lemma}

\begin{proof}
Since $p^{t}< n \leq p^{t+1}$, we have  $np^{t+m+2}\leq p^{2t+m+3}$ and
$\frac{q^n-q}{q+1}\geq q^{n-2}=q_0^{p^m(n-2)}\geq p^{p^m(p^{t}-1)}$. Hence, it suffices to prove that
$p^{2t+m+3}\leq p^{p^{m}(p^t-1)}$, which is equivalent to prove that $2t+m+3\leq p^{m}(p^t-1)$.
Notice that $2t+m+3\leq 2^m(2^t-1)$ holds unless one of the following cases occurs:
$(i)$ $t=0$; $(ii)$ $t=1$ and $m=0,1,2$; $(iii)$ $t=2$ and $m=0,1$; $(iv)$ $t=3$ and $m=0$. 
Direct computations in each case yield the statement.
\end{proof}

\begin{lemma}\label{LU2}
Let $L\in\{\PSL_n(q), \PSU_n(q)\}$, where $n\geq 5$.
Let $g\in \Aut L$ be a $p$-element such that $g^2\neq 1$ and
let $G=\langle g,L\rangle$.
Then $\phi(g)$ is almost cyclic for some  cross-characteristic projective irreducible representation $\phi$ of $G$
that is non-trivial on $L$ if and only if $G=\PSU_5(2).2$, $\dim \phi=10$ and $|g|=16$.
\end{lemma}

\begin{proof}
Recall (see Lemma \ref{md6}) that $\dim \phi \geq \frac{q^n-1}{q-1}-2$ if $L=\PSL_n(q)$, whereas $\dim\phi\geq 
\frac{q^n-q}{q+1}$ if $L=\PSU_n(q)$ and $n$ is odd, and $\dim\phi\geq \frac{q^n-1}{q+1}$ if $L=\PSU_n(q)$ and $n$ is 
even. Also, notice that $\frac{q^n-q}{q+1}\leq \frac{q^n-1}{q+1}\leq \frac{q^n-1}{q-1}-2$.

If $p\geq 3$, then $\eta_p(G)\leq p^{t+m+1}$, where $t$ and $m$ are defined as above. By Lemma \ref{LU} we get 
$n|g|\leq 
np^{t+m+1} \leq \dim \phi$. Thus $\phi(g)$ is not almost cyclic, by Lemma \ref{uu3}. If $p=2$, then $\eta_2(G)\leq 
2^{t+1}\cdot 2^{t+m+2}$.
Likewise, we may apply Lemma \ref{LU} obtaining that $\phi(g)$ is not almost cyclic, unless $L$ is one of the following 
groups:
$\PSL_5(2)$, $\PSU_5(2)$,
$\PSL_6(2)$, $\PSU_6(2)$, $\PSU_7(2)$, $\PSU_8(2)$ and $\PSU_9(2)$. Thus, we need to examine these groups in detail. 
Moreover, in order to bound efficiently $\alpha(g)$ we make use of the package GAP whenever necessary.

1) $L=\PSL_5(2)$. In this case  $|g|\leq 16$ and $\dim \phi \geq 29$.
If $|g|=4$, then $\alpha(g)\leq 3$; if $|g|=8$, then $\alpha(g)=2$. Thus, if $|g| \leq 8$, $\phi(g)$ is not almost 
cyclic, by Lemma  \ref{uu3}.
If $|g|=16$, then $\alpha(g)=2$. This means that we have to examine the representations $\phi$ such
that $29 \leq \dim \phi\leq 30$. Using the GAP package, again we see that $\phi(g)$ is not almost cyclic.

2) $L=\PSU_5(2)$. In this case  $|g|\leq 16$ and $\dim \phi \geq 10$.
If $|g|=4$, then $\alpha(g)\leq 3$. Thus $\phi(g)$ is not almost cyclic by Lemma  \ref{uu3}. 
If $|g|=8$, then $\alpha(g)=2$, and hence we need to examine the representations $\phi$ such that $10\leq \dim \phi 
\leq 
14$. If $|g|=16$, then $\alpha(g)=2$ and we need to examine the representations $\phi$ such
that $10 \leq \dim \phi\leq 30$. Using the GAP package, we see that $\phi(g)$ is almost cyclic if and only if 
$G=\PSU_5(2).2$, $\dim \phi=10$ and $|g|=16$, as claimed.

3) $L=\PSL_6(2)$. In this case  $|g|\leq 16$, $\alpha(g)\leq 3$ and $\dim \phi \geq 61$. Thus, by Lemma  \ref{uu3}, 
we conclude that $\phi(g)$ is not almost cyclic.

4) $L=\PSU_6(2)$. In this case  $|g|\leq 16$ and $\dim \phi \geq 21$.
If $|g|=8$ or $16$, then $\alpha(g)=2$, whereas if $|g|=4$, then $\alpha(g)\leq 3$.
Thus, if $|g|=4$ or $8$, by Lemma \ref{uu3} $\phi(g)$ is not almost cyclic.
If $|g|=16$ ($g \not \in L$), we need to examine the representations $\phi$ such that $21 \leq \dim \phi\leq 30$.
Using the GAP package, we find that $\phi(g)$ is never almost cyclic.

5) $L=\PSU_7(2)$. In this case  $|g|\leq 16$ and $\dim \phi \geq 42$. If $|g|=4$, then $\alpha(g)\leq 4$, whereas, if 
$|g|=8$ or $16$, then $\alpha(g)=2$. In both cases, by Lemma \ref{uu3}, $\phi(g)$ is not almost cyclic.

6) $L=\PSU_8(2)$. In this case  $|g|\leq16$, $\alpha(g)\leq 4$ and $\dim\phi \geq 85$.
By Lemma \ref{uu3}, $\phi(g)$ is not almost cyclic.

7) $L=\PSU_9(2)$. In this case  $|g|\leq 32$ and $\dim \phi
\geq 170$. If $|g| \leq 16$, then $\phi(g)$ is not almost cyclic by Lemma \ref{uu3}. If $|g|=32$, then we may consider 
$g^{16}$. The latter belongs to $L$, and $L$ is generated by $9$ conjugates of $g^{16}$. By Lemma \ref{uu3},  $\dim 
\phi 
\leq 9d$, where $d$ equals the dimension of an eigenspace of $\phi(g^{16})$. Suppose that $\phi(g)$ is almost cyclic. 
Then it readily follows that either $1$ or $-1$ appears as an eigenvalue of $\phi(g^{16})$ with multiplicity $\leq 16$. 
Thus we are led to conclude that $\dim \phi \leq 9d \leq 9\cdot16 < 170$, a contradiction.
\end{proof}

\begin{lemma}\label{SL34}
Let $L\in \{\PSL_3(q),\PSL_4(q)\}$, with  $L \neq \PSL_3(2),$ $\PSL_3(4), \PSL_4(2)$.
Let $g\in \Aut L$ be a $p$-element such that $g^2\neq 1$ and
let $G=\langle g,L\rangle$.
Then $g$ is not almost cyclic in any cross-characteristic projective irreducible representation $\phi$ of $G$
that is non-trivial on $L$.
\end{lemma}

\begin{proof}
Let $L=\PSL_3(q)$. If $q = 3$, then $|g|=3$, $\alpha(g)\leq 3$ and $\dim \phi \geq 11$, whence 
$11\leq 3(3-1)=6$. Thus $\phi(g)$ cannot be almost cyclic. Now, assume that $q\geq 5$.

Let $t$ and $m$ be defined as above and suppose that $p>2$. Then $n = 3$ implies $t = 0$,
whence $\eta_p(G)\leq p^{m+1}$.
Therefore $3\eta_p(G)\leq 3p^{m+1}\leq p^{m+2}$, whilst $\dim \phi \geq q^2+q-1\geq q^2\geq p^{2p^m}$. Thus it is 
enough 
to show
that $m+2 \leq 2 \cdot 3^m\leq  2p^m$. This clearly holds  for all $m\geq 0$.
If $p=2$, then $\eta_2(G)\leq 2^{m+3}$. Hence (see Proposition \ref{GS2}) we need to show that
$4\eta_2(G)\leq 2^{m+5}\leq 2^{2^{m+1}}\leq q^2+q-1$.
Now, $m+5\leq 2^{m+1}$ provided $m > 1$.
For $m=0,1$, we get the inequality $4\cdot2^{m+3} \leq 64\leq q^2+q-1$, which holds for any even $q>4$.

Next, let $L=\PSL_4(q)$. 
If $q=3$, then $|g|\leq 9$ and $\dim \phi \geq 26$.
If $|g|=3$, then $\alpha(g)\leq 4$, whereas, if $|g|=9$, then $\alpha(g)=2$. In both cases, $\phi(g)$ cannot be almost 
cyclic, by Lemma \ref{uu3}. Now, assume that $q\geq 5$.

If $p>3$, then $\eta_p(G)\leq p^{m+1}$.
Hence $4\eta_p(G)\leq 4 p^{m+1}\leq p^{m+2}\leq p^{3p^m}\leq q^3+q^2+q-1\leq \dim \phi$, and we are done.
If $p=3$, then  $\eta_3(G)\leq 3^{m+2}$, and hence $4\eta_3(G)\leq 4\cdot 3^{m+2}\leq 3^{m+4}\leq 3^{3^{m+1}}\leq
\dim \phi$ for all $m\geq 1$. If $m=0$ (and $q>3$), we still have  $4\eta_3(G)\leq 36 \leq q^3\leq \dim \phi$.
Finally, let  $p=2$. Then $\eta_2(G)\leq 2^{m+3}$. By Proposition \ref{GS2} $\alpha(g)\leq 6$, whence $6\eta_2(G)\leq 
2^{m+6}\leq 2^{3\cdot 2^m} \leq \dim \phi$ for all $m\geq 2$. Furthermore, if $m=0,1$ (and $q>4$), then  
$6\eta_2(G)\leq 
96\leq 8^3\leq \dim \phi$.

Finally, let  $L=\PSL_4(4)$. In this case $|g|=4,8$, $\alpha(g)\leq 3$ and 
$\dim 
\phi\geq 83$, and hence $\phi(g)$ cannot be almost cyclic, by Lemma \ref{uu3}.
\end{proof}

\begin{lemma}\label{u34}
Let $L\in \{\PSU_3(q), q>4, \PSU_4(q),q>3\}$.
Let $g\in \Aut L$ be a $p$-element such that $g^2\neq 1$ and
let $G=\langle g,L\rangle$.
Then $g$ is not almost cyclic in any cross-characteristic projective irreducible representation $\phi$ of $G$
that is non-trivial on $L$.
\end{lemma}

\begin{proof}
Let $L=\PSU_3(q)$. If $p>2$, then  $\eta_p(G)\leq p^{m+1}$.
Since $3\eta_p(G)\leq 3p^{m+1}$ and $\dim \phi \geq q^2-q\geq 3q\geq 3 p^{p^m}$, it suffices to show
that $m+1\leq p^m$, which clearly holds for all $m\geq 0$.
If $p=2$, then $\eta_2(G)\leq 2^{m+3}$. Hence, it will be enough to show that
$4\eta_2(G)\leq 2^{m+5}\leq 2^{1+2^m}\leq  2q\leq q^2-q$.
Now, $m+4\leq 2^{m}$ provided $m\geq 3$.
If $0\leq m \leq 2$, then $4\eta_2(G)\leq 128\leq q^2-q$ for all $q\geq 16$.
If $q=8$, then $4\eta_2(G)=32\leq \dim\phi$. Therefore, we are done.

Now, let $L=\PSU_4(q)$. If $p>3$, then $\eta_p(G)\leq p^{m+1}$.
Whence $4\eta_p(G)\leq 4 p^{m+1} \leq 4 p^{2p^m}\leq 4q^2\leq q^3-q^2+q-1
 \leq \dim \phi$ (note that $m+1 \leq 2p^m$ for all $m\geq 0$).
If $p=3$, then  $\eta_3(G)\leq 3^{m+2}$. Hence, $4\eta_3(G)\leq 4\cdot 3^{m+2}\leq 4\cdot 3^{2\cdot 3^{m}}\leq
\dim \phi$ for all $m\geq 0$.
If $p=2$, then $\eta_2(G)\leq 2^{m+3}$. It follows that $6\eta_2(G)\leq 6\cdot 2^{m+3}
\leq 6\cdot 2^{2^{m+1}}\leq 6q^2\leq \dim \phi$
holds for all $m\geq 1$. If $m=0$ (and $q>4$), then $6\eta_2(G)\leq 48\leq q^3-q^2+q-1$.

Finally, let $L=\PSU_4(4)$. In this case $\eta_2(G)=16$, $\alpha(g)=2$  and $\dim \phi \geq 51$; hence $\phi(g)$ 
cannot be almost cyclic, by Lemma \ref{uu3}.
\end{proof}

\begin{lemma}\label{SP43}
Let $L=\PSp_{2n}(q)$, where $n\geq 2$ and $q$ is odd. 
Let $1\neq g\in \Aut L$ be a $p$-element and
let $G=\langle g,L\rangle$.
Let $\phi$ be a cross-characteristic projective irreducible representation  $G$
that is non-trivial on $L$. Then $\phi(g)$ is almost cyclic if and only if $L=\PSp_4(3)$ and the following occurs:
\begin{itemize}
 \item[(1)] $\ell\neq 2,3$: $g=3a,3b$ and $\dim \phi=4$; $g=3c$ and $\dim \phi=6$;
$g=3d$ and $\dim \phi=4,5$; $|g|=9$ and $\dim \phi=4,5,6$;
\item[(2)] $\ell=2$: $g=3a,3b$ and $\dim \phi=4$; $g=3c$ and $\dim \phi=6$; $g=3d$ and $\dim \phi=4$;
$|g|=9$ and $\dim =4,6$.
\end{itemize}
\end{lemma}

\begin{proof}
First, note that $\eta_p(G)\leq p^{t+m+1}$, where $p^t < 2n\leq p^{t+1}$.
Thus, we aim to show that $2n\eta_p(G)\leq 2n p^{t+m+1}\leq \frac{q^n-1}{2}$, that is, 
$4n p^{t+m+1}\leq q^n-1$.
Now,  $4n p^{t+m+1}\leq 2p^{2t+m+2}\leq  p^{2t+m+3}-1$
and $q^n-1\geq p^{np^m}-1$. Therefore, it will suffice to consider the inequality
$2t+m+3\leq np^{m}$.

If $g \in L$, we may refer to \cite{DMZ10}. In particular, by \cite[Theorem 1.1 and Lemma 4.2]{DMZ10}, 
$\phi(g)$  is almost cyclic if and only if $L=\PSp_4(3)$ and (1) or (2) holds.
So, suppose that $g \not \in L$. This implies in particular that we may assume $m>0$.

If $t\geq 1$, then $2t+m+3\leq 
p^{t+m-1}\leq np^{m}$, unless $t=m=1$ and $p=3,5$.
For $p=3$, we get $3< 2n \leq 9$ and
$2n\eta_3(G)\leq 54n \leq \frac{3^{3n}-1}{2}$.
For $p=5$, we get  $5< 2n \leq 25$ and
$2n\eta_5(G)\leq 250n \leq \frac{5^{5n}-1}{2}$. Thus we are done.

Finally, suppose that $t=0$, i.e. $2n\leq p$. In this case, we have to consider the inequality
$2n\eta_p(G)\leq p^{m+2}\leq \frac{p^{np^m}-1}{2}$.
Since $2p^{m+2}\leq p^{m+3}-1\leq p^{np^m}-1$ is equivalent to
$m+3\leq n p^m$, we only need to observe that $m+3\leq 2\cdot 3^m$ holds for all $m\geq 1$, thus proving the statement.
\end{proof}

 \section{Some low-dimensional classical groups}\label{low}
 
In this section we deal with the low-dimensional simple groups $L$ which were left out in either or both the two
previous sections. These groups, which need a separate treatment, are the following: $\PSL_3(2), \PSL_3(4), \PSL_4(2),
\PSU_3(3), \PSU_4(2) \cong \PSp_4(3), \PSU_4(3)$. Observe that, with the exception of $\PSU_3(3)$  they all have an
exceptional Schur multiplier (by this, we mean that the order of the multiplier is divisible by the defining
characteristic of the group in question).  As always, $G=\langle g,L\rangle$, where $g\in \Aut L$. When needed, we use
the notation of the GAP package to denote $G$ and the conjugacy class of $g$. Also, in view of
Lemma  \ref{news2a}, we may assume $g^2\neq 1$.

\subsection{$L=\PSL_3(2) \cong \PSL_2(7)$}\label{L32}

Here $|g| \in \{ 3, 4, 7, 8 \}$ and $\alpha(g)=2$.
Using the GAP package, we find that $\phi(g)$ is almost cyclic if
and only if one of the following occurs:
\begin{itemize}
\item $|g|=3$ and  either $\ell\neq 7$ and  $\dim \phi =3,4$, or $\ell=7$ and  $\dim \phi =2,3,4$;
\item$|g|=4$ and either $\ell\neq 7$ and $\dim \phi=3,4$,  or $\ell=7$ and $\dim \phi=2,3,4,5$;
\item $|g|=7$ and  either $\ell\neq 3,7$ and $\dim \phi =3,4, 6,7,8$, or
$\ell=3$ and  $\dim \phi =3,4, 6,7$, or
$\ell=7$ and  $\dim \phi =2,3,4,5,6,7$;
\item $|g|=8$ and either  $\ell\neq 7$ and $\dim \phi=6,7,8$,
or $\ell=7$ and $\dim \phi=2,3,4,5,6,7$.
\end{itemize}

\subsection{$L=\PSL_3(4)$}\label{L34}
 
Here $|g| \in \{3, 4, 5, 7, 8 \}$.
If $g \not \in L$ and has order $3$, then $\alpha(g)\leq 3$; otherwise $\alpha(g)=2$.
 
Suppose first that $g$ is a $2$-element. Then, using the GAP and MAGMA packages we find that $\phi(g)$ is almost cyclic
if and only if one of the following occurs:
\begin{itemize}
\item $\ell\neq 3,7$: $|g|=8$ and either $G=L.2_1$ and $\dim \phi=6$ or $G=L.2_3$ and $\dim \phi=8$;
\item $\ell=3$: $|g|=4$, $G=L$ and $\dim \phi=4$; $|g|=8$ and either $G=L.2_1$ and $\dim \phi=6$, or
$G=L.2_2$ and $\dim \phi=4,6$, or $G=L.2_3$ and $\dim \phi=6,8$;
\item $\ell=7$: $|g|=8$ and either $G=L.2_1$ and $\dim \phi=6$, or $G=L.2_3$ and $\dim \phi=8$, or $g\in 8a$ and
$\dim \phi=10$.
\end{itemize}

Next, suppose that $|g| = 3$. Then, again using the GAP and MAGMA packages we find that $\phi(g)$ is almost cyclic if and only if: 
\begin{itemize}
\item $\ell=3$, $g\in L$ and $\dim \phi=4$.
\end{itemize}

Finally, suppose that $|g| = 5,7$. Here in principle we have to examine representations of dimension $4$, $6$, $8$
and $10$, where 
$4$-dimensional representations only occur for $\ell = 3$. Furthermore, the Sylow $5$-subgroups as well as the
Sylow $7$-subgroups are cyclic; hence, when $\ell = 5$ and $\ell = 7$, respectively, 
we can refer to \cite[Lemma 2.13 and Corollary 2.14] {DMZ12}. 
As a result, we get that $\phi(g)$ is almost cyclic if and only if one of the following occurs:
\begin{itemize}
\item $|g|=5$ and either $\ell=3$ and $\dim \phi=4$, or $\dim \phi = 6$;  
\item $|g|=7$ and either $\ell=3$ and $\dim \phi=4$, or $\dim \phi = 6,8$.
\end{itemize}

\subsection{$L=\PSL_4(2)$}\label{L42}

Here $|g|\in \{  3, 4, 5, 7, 8 \}$ and
either $\dim \phi \in \{7,8,13,14\}$  or $\dim \phi \geq 16$.
If $|g|=3$, then $g \in L$ and $\alpha(g)\leq 4$, and hence we only have to examine the representations $\phi$ of
dimension $7$ and $8$.
If $|g|=4$ and $g \not \in L$, then $\alpha(g)\leq 3$, and hence we only have to examine the representations $\phi$ such that  $7\leq \dim \phi\leq 9$.
In all the other cases, $\alpha(g)=2$.
Using the GAP package, we find that $\phi(g)$ is almost cyclic if and only if one of the following occurs:
\begin{itemize}
 \item $\dim  \phi=7$  and either $g \in 3a, 5$, or  $|g|=4$ and $g \not \in L$;
 \item $\dim  \phi=7,8$ and $|g|=7,8$.
\end{itemize}

\subsection{$L=\PSU_3(3)$}\label{U33}

Using the GAP package, we find that  $\alpha(g)=2$ for $|g|=7,8$ and $\alpha(g)\leq 3$ for $|g|=3,4$.
Suppose first that $p\neq 3$. The following holds.

If $G=L$, then $\phi(g)$ is almost cyclic if and only if:
\begin{itemize}
\item $|g|=7,8$, $\dim \phi=6,7$ and $\ell\neq 2,3$;
\item $|g|=7,8$, $\dim \phi=6$ and $\ell=2$.
\end{itemize}

If $G=L.2$, then $p=2$ and $\phi(g)$ is almost cyclic if and only if:
\begin{itemize}
\item $|g|=8$, $\dim \phi=6,7$ and $\ell\neq 2,3$;
\item $|g|=8$, $\dim \phi=6$ and $\ell=2$.
\end{itemize}

Next, assume that $p=3$.
In this case $|g|=3$, $\alpha(g)\leq 3$  and $\dim \phi \geq 6$. So, we only need to examine the representations
$\phi$ of degree $6$. Note that $g \in L$.
Again, using the GAP package, we obtain that $\phi(g)$ is not almost cyclic.

\subsection{$L=\PSU_3(4)$}\label{U34}

This group was left out by assumption in Lemma \ref{u34}. In this case $\eta_2(G)= 16$, $\alpha(g)=2$  and $\dim \phi\geq 12$.
If $|g|=4$, then $\phi(g)$ cannot be almost cyclic, by Lemma \ref{uu3}. If $|g|=8$, we only need to examine the representations $\phi$ such that $12\leq \dim \phi\leq 14$. If $|g|=16$, we only need to examine the representations $\phi$ such that
$12\leq \dim \phi\leq 30$.
Using the GAP package, we find that $\phi(g)$ is almost cyclic if and only if
$|g|=16$ and $\dim \phi=12$.

\subsection{$L=\PSU_4(2) \cong \PSp_4(3)$}\label{U42}

In this case, $|g| \in \{ 3, 4, 5, 8, 9 \}$.
If $|g|=3$, then $g \in L$ and $\alpha(g)\leq 4$.
If $|g|=4$ and $g \in L$, then $\alpha(g)=2$ (in which case we only have to examine the representations $\phi$ such that
$4\leq \dim \phi\leq 6$).
If $|g|=4$ and $g \not \in L$, then $\alpha(g)\leq 3$ (in which case we only have to examine the representations $\phi$ such that
$4\leq \dim \phi\leq 9$). If $|g|=8$, then $g \not \in L$ and $\alpha(g)=2$ (hence we only have to examine the representations $\phi$ such that
$4\leq \dim \phi\leq 14$). Finally, if $|g| = 5,9$, then $\alpha(g)=2$.
Using the GAP package we find that $\phi(g)$ is almost cyclic if
and only if one of the following cases occurs:

\begin{itemize}
 \item $G=L$ and
\begin{itemize}
\item[(a)] $\ell\neq 3$, $g \in 3a,3b,3d$ or $|g|=4,5, 9$ and $\dim \phi=4$;
\item[(b)] $\ell=3$, $g \in 3a,3b, 4a$ or $|g|=5,9$ and $\dim \phi=4$;
\item[(c)] $\ell\neq 2,3$, $g\in 3d, 4b$ or $|g|=5,9$ and $\dim \phi=5$;
\item[(d)] $\ell= 3$, $g \in 3c,3d,4b$ or $|g|=5,9$ and $\dim \phi=5$;
\item[(e)] $\ell \neq 3$, $g \in 3c$ or $|g|=5,9$ and $\dim \phi=6$.
\end{itemize}

\item $G=L.2$ and
\begin{itemize}
\item[(f)] $\ell=3$, $g\in 4d$ or $|g|=8$ and $\dim \phi=4$;
\item[(g)] $\ell=3$, $|g|=4,8$ and $\dim \phi=5$;
\item[(h)] $\ell\neq 3$, $|g|=4,8$ and $\dim \phi=6$.
\end{itemize}
\end{itemize}
Clearly, if we view $L$ as $\PSU_4(2)$ we must add the extra condition $\ell\neq 2$ to items (a), (e) and (h), 
whereas if we view $L$ as $\PSp_4(3)$ we must ignore items (b), (d), (f) and (g). Also observe that in item (c) the 
assumption $\ell \neq 2$ is due to the fact that $\PSp_4(3)$ does not have $2$-modular irreducible representations of 
degree $5$.

\subsection{$L=\PSU_4(3)$}\label{U43}

In this case, $\dim\phi \geq 6$. Observing that if $g \not \in L$, then $g$ is a $2$-element, and using the GAP package,
we get the following evidences for the values of $\alpha(g)$ and the almost cyclicity condition on $\phi(g)$: 
$$\begin{array}{c|c|c|c||c|c|c|c }
|g| &  & \alpha(g) & \dim \phi\leq & |g| &  & \alpha(g) & \dim \phi\leq \\\hline
3 & g \in L & \leq 4 &  4(3-1)=8 & 7 & g \in L & =2 &  2(7-1)=12 \\
4  & g \in L & =2 &  2(4-1)=6 & 8 & g\in L & =2 &  2(8-1)=14 \\
4 & g \not \in L &  \leq 4 &  4(4-1)=12 & 8 & g\not\in L & =2 &  2(8-1)=14 \\
5 & g \in L & =2 &  2(5-1)=8 & 9 & g \in L & =2 &  2(9-1)=16
\end{array}$$

It follows that we only have to examine the representations $\phi$ such that $6\leq \dim\phi \leq 16$. It turns out that if either $\ell=0$ or $(p,\ell)=1$,
then  $\phi(g)$ is almost cyclic if and only if $\dim \phi=6$ and
\begin{itemize}
\item $G=L$, $g \in 3b$ or $|g|=5,7,8,9$;
\item $G=L.2_2$,  $g\in 4d$ or $|g|=8$ and $\ell\neq 2,3$.
\end{itemize}

Next, suppose that $\ell = p$. Then, for $\ell=2,5,7$, we are left to examine the following possibilities:

\begin{itemize}
\item[(1)] $\ell=2$, $G\in \{L, L.2_2\}$,  $|g|=4,8$ and $\dim \phi=6$;
\item[(2)] $\ell=2$, $G=L.2_1$, $|g|=4,8$ and $\dim \phi=12$;
\item[(3)] $\ell=5$, $G=L$,  $|g|=5$ and $\dim \phi=6$;
\item[(4)] $\ell=7$, $G=L$,  $|g|=7$ and $\dim \phi=6$.
\end{itemize}
Considering (1), we see that if $|g| = 8$, then, since $\dim \phi=6$, $\phi(g)$ must have a
single non-trivial block of size $\geq 5$. Thus $\phi(g)$ is almost cyclic. So, assume that $|g| = 4$. 
Assume first that $g\in L$. There are two classes of elements of order $4$ in $L$, namely $4a,4b$. 
If $g\in 4a$, w.l.o.g
we may assume that $g$ lies in a maximal subgroup $M$ of $L$ isomorphic to $\PSU_3(3)$. We find that $\phi$ 
restricted to
$M$ remains irreducible and $\phi(g)$ is the sum of two Jordan blocks of size $3$. So $\phi(g)$ is not almost cyclic.
If $g\in 4b$, w.l.o.g we may assume that $g$ lies in a maximal subgroup $N$ of $L$ isomorphic to $A_7$. We find that
$\phi$ restricted to $N$ remains irreducible and $\phi(g)$ is the sum of two Jordan blocks of size $4$ and $2$. So again
$\phi(g)$ is not almost cyclic. 

Next, assume that $g\not\in L$ and $\langle L,g\rangle=L.2_2$.
There are two classes of such elements, namely $4c,4d$.
If $g\in 4c$, w.l.o.g
we may assume that $g$ lies in a maximal subgroup $P$ of $L$ isomorphic to $\PSU_3(3).2$.
We find that $\phi$ restricted to
$P$ remains irreducible and $\phi(g)$ is the sum of two Jordan blocks of size $3$. So $\phi(g)$ is not almost cyclic.
If $g\in 4d$, w.l.o.g we may assume that $g$ lies in a maximal subgroup $Q$ of $L$ isomorphic to $S_7$. We find that
$\phi$ restricted to $Q$ remains irreducible and $\phi(g)$ is the sum of two Jordan blocks of size $1$ and one of
size $4$. So, $\phi(g)$ is almost cyclic.

The instances in (2) can be ruled out, as a $2$-modular $12$-dimensional representation of $G$ when 
restricted to $L$ splits into two $6$-dimensional components intertwined by $g$. Now, $\phi(g)$ is almost cyclic if and 
only if
$g^2$ is cyclic on each component, but this cannot be by \cite[Proposition 2.14]{DMZ0}.

Finally, observe that the Sylow $5$-subgroups as well as the Sylow $7$-subgroups are cyclic. 
It follows that in instances (3) and (4), by \cite[Lemma 2.13 and Corollary 2.14] {DMZ12}, 
$\phi(g)$ is almost cyclic (in fact, cyclic in instance (4)).

We conclude that, for any value of $\ell$, $\phi(g)$ is almost cyclic if and only if $\dim \phi=6$ and one of the following occurs:
\begin{itemize}
\item $G=L$ and either $g \in 3b$ or $|g|=5,7,8,9$;
\item $G=L.2_2$ and either $g \in 4d$ or $|g|=8$.
\end{itemize}

\begin{rem}
Note that in the following cases $\phi|_L$ is an irreducible Weil representation of $L$, and $g$ is a semisimple element. Thus, they have been already described in \cite{DMZ12}.
\begin{itemize}
\item[(1)] $G=\PSL_3(2)$, $\dim \phi=5$, $\ell=7$ and $|g|=7$;
\item[(2)] $G=\PSL_3(2)$, $\dim \phi=7$ and  $|g|=7$;
\item[(3)] $G=\PSU_3(3)$, $\dim \phi=6$ and $|g|=7,8$;
\item[(4)] $G=\PSU_3(3)$, $\dim \phi=7$, $\ell\neq 2$ and $|g|=7,8$;
\item[(5)] $G=\PSU_4(2)$, $\dim\phi =5$ and $g \in 3d$ or $|g|=5,9$;
\item[(6)] $G=\PSU_4(2)$, $\dim \phi=5$, $\ell = 3$ and $g \in 3c$;
\item[(7)] $G=\PSU_4(2)$, $\dim\phi =6$, $\ell\neq 3$ and $g \in 3c$ or $|g|=5,9$;
\item[(8)] $G=\PSp_4(3)$, $\dim\phi =4$ and $|g|=4,5$;
\item[(9)] $G=\PSp_4(3)$, $\dim\phi =5$, $\ell \neq 2$ and $g \in 4b$ or $|g|=5$.
 \end{itemize}
\end{rem}

 \section{The groups $\Sp_{2n}(q)$, $n > 1$, $q$ even; $\Omega_{2n+1}(q)$, $n> 2$, $q$ odd; 
 $\POmega_{2n}^{\pm}(q)$, $n> 3$}\label{otherclass}
 
As shown in \cite[Theorem 2.16]{GMPS}, if $L$ is one of the simple groups 
$\Sp_{2n}(q)$ ($n\geq 2$ and $q$
even), $\Omega_{2n+1}(q)$ ($n\geq 3$ and $q$ odd) or $\POmega_{2n}^{\pm }(q)$ ($n\geq 4$), 
then $|g|\leq \frac{q^{n+1}}{q-1}$  for all $g \in \Aut L$.

\begin{lemma}\label{SP2}
Let $L=\Sp_{2n}(q)$, where $n\geq 2$, $q$ is even and $(n,q)\neq (4,2)$. Let $g \in \Aut L$, where $g^2\neq 1$, and let $G=\langle
L,g\rangle\subseteq \Aut L$.
Let $\phi$  be a cross-characteristic projective irreducible representation $\phi$ of $G$
that is non-trivial on $L$.
Then $\phi(g)$ is  almost cyclic if and only if one of the following occurs:
\begin{itemize}
\item[(1)] $G=\Sp_4(4)$, $|g|=17$ and $\dim \phi=18$.
\item[(2)] $G=\Sp_6(2)$ and
\begin{itemize}
\item[(a)] $\dim \phi=7$ and either $g \in 3a, 4c$, or $|g|=5,7,8,9$;
\item[(b)] $\dim \phi=8$ and $|g|=7,8,9$.
\end{itemize}
\end{itemize}
\end{lemma}

\begin{proof}
First, recall that $\dim \phi \geq \frac{(q^n-1)(q^n-q)}{2(q+1)}$ and
$\alpha(g)\leq 2n+1$ for any $g \in \Aut L$.
Suppose that $\phi(g)$ is almost cyclic. By Lemma \ref{uu3} this implies that
$\frac{(q^n-1)(q^n-q)}{2(q+1)}\leq (2n+1)(|g|-1)\leq (2n+1)\frac{q^{n+1}-q+1}{q-1}$.
The above inequality holds if and only if $(2n,q)\in \{(4,4),(4,8), (6,2),(6,4),(8,2), (10,2),(12,2),(14,2)\}$. 
Thus, almost cyclicity can only occur for these values of the pair $(2n,q)$.
However, it is to check that, unless $(2n,q) = (4,4),(6,2)$,  $\phi(g)$ cannot be almost cyclic, by Lemma  \ref{uu3}.

If $L=\Sp_4(4)$, then $\dim \phi\geq 18$ and $|g|\in \{3,4,5,8,17 \}$. 
If $|g|=4,8,17$, then $\alpha(g)=2$; if
$|g|=3,5$, then $\alpha(g)\leq 3$. If $|g| < 17$, then $18 > \alpha(g)|g|$, and hence $\phi(g)$ is not almost cyclic. 
If $|g|=17$, then $g \in L$, and using the
GAP package, we see that $\phi(g)$ is almost cyclic if and only if $\dim \phi=18$. Whence item (1) of the statement.

If $L=\Sp_6(2)$, then $|g| \in \{ 3, 4, 5, 7, 8, 9  \}$ and $\dim \phi\geq 7$.
If $|g|=3$, then $\alpha(g)\leq 4$;
if $|g|=4$, then $\alpha(g)\leq 3$; finally, 
if $|g|\geq 5$, then $\alpha(g)=2$.
Thus $\phi(g)$ is not almost cyclic, unless possibly when $\dim\phi \leq 16$.
Using the GAP package,  we see that $\phi(g)$ is almost cyclic if and only if one of the instances listed in item (2) 
of the statement occurs.
\end{proof}

\begin{remar}\label{Sp42d}
The group $\Sp_4(2)$ is not simple. However, the commutator subgroup $L = \Sp_2(2)'$ is simple and isomorphic to 
$\PSL_2(9)$.  Using the package GAP and assuming that $g$ is not involution, we find that 
$\phi(g)$ is almost cyclic if and only if $ G = L$ and the following holds:
\begin{itemize}
\item[(1)] $\dim \phi= 3,4,5$, $\ell\neq 2,3$ and $|g|=3,4,5$;
\item[(2)] $\dim \phi=6$, $\ell\neq 2$ and $|g|=5$;
\item[(3)] $\dim \phi=3,4$, $\ell=2$ and  $|g|=3,4,5$;
\item[(4)] $\dim \phi = 2,3,4$,  $\ell=3$ and $|g|=3,4,5$.
\end{itemize}
\end{remar}

\begin{lemma}\label{O1} 
Let $L=\Omega_{2n+1}(q)$, where $n\geq 3$ and $q$ is odd, or $L=\POmega_{2n}^-(q)$, where $n\geq 4$.
Let $g \in \Aut L$, where $g^2\neq 1$, and let $G=\langle L,g\rangle\subseteq \Aut L$.
Then $\phi(g)$ is not almost cyclic in any cross-characteristic projective irreducible representation $\phi$ of $G$
that is non-trivial on $L$.
\end{lemma}

\begin{proof}
Suppose that $L=\Omega_{2n+1}(q)$, and assume that $\phi(g)$ is almost cyclic.
By Lemma \ref{md6}, if $q=3$ and $n\geq 4$, then $\dim \phi \geq \frac{(3^n-1)(3^n-3)}{8}$,
whence the bound
$$\frac{(3^n-1)(3^n-3)}{8}\leq (2n+1) \frac{3^{n+1}-2}{2},$$
which holds if and only if $n=4$.
If $q>3$, then, again by Lemma \ref{md6}, $\dim \phi \geq \frac{q^{2n}-1}{q^2-1}-2$,
whence the bound
$$\frac{q^{2n}-2q^2+1}{q^2-1}\leq (2n+1) \frac{q^{n+1}-q+1}{q-1},$$
which holds if and only if $n=3$ and $q=5,7$.
By the above, we are left to examine the cases where $(n,q)\in\{(3,3),(3,5),(3,7),(4,3)\}$. 
In each of these cases, it is routine to check that  $\phi(g)$ is not almost cyclic, by Lemma \ref{uu3}.

Now, suppose that $L=\POmega_{2n}^-(q)$ and assume that $\phi(g)$ is almost cyclic.
By Lemma  \ref{md6}, (8), 
$\dim \phi \geq \frac{(q^n+1)(q^{n-1}-q)}{q^2-1}-1$, provided $(n,q) \neq 
(4,2),(4,4),(5,2),$ $(5,3)$. In the first case,
the following bound must be met:
$$\frac{(q^n+1)(q^{n-1}-q)-(q^2-1)}{q^2-1}  \leq 2n \frac{q^{n+1}-q+1}{q-1}.$$
However, this only holds if either $n=4$ and $q\leq 8$ or $n=6,7$ and $q=2$.
Thus, we are left to examine the following possible exceptional cases: 
$(n,q) = (4,2),$ $(4,3),(4,4), (4,5),(4,7), (5,2),(5,3),(6,2),(7,2)$.
In each individual case, from the available data we get $\dim \phi > \alpha(g)(|g|-1)$, a contradiction.
\end{proof}

\begin{lemma}\label{O3}
Let $L=\POmega_{2n}^+(q)$, where $n\geq 4$. 
Let $g \in \Aut L$, where $g^2\neq 1$, and set $G=\langle L,g\rangle\subseteq \Aut L$.
Let $\phi$ be a cross-characteristic projective irreducible representation of $G$
non-trivial on $L$. Then $\phi(g)$ is almost cyclic if and only if  
$L= \Omega_8^+(2)$, $\dim \phi=8$ and one of the following occurs:
\begin{itemize}
\item $G = L$ and either $g \in  3a, 5a, 8b$ or $|g| = 7, 9$;
\item $G=\SO_8^+ (2)$ and either $g\in  4f$ or $|g| = 8$.
\end{itemize}
\end{lemma}

\begin{proof}
Suppose that $\phi(g)$ is almost cyclic.
If $q\leq 3$ and $(n,q)\neq (4,2)$, then $\dim \phi \geq \frac{(q^n-1)(q^{n-1}-1)}{q^2-1}$.
It follows that the bound
$$\frac{(q^n-1)(q^{n-1}-1)}{q^2-1} \leq 2n \frac{q^{n+1}-q+1}{q-1}$$
must be met. This happens  if and only if either $n\leq 5$ or $q=2$ and $n=6,7$.
If $q>3$, then $\dim \phi \geq \frac{(q^n-1)(q^{n-1}+q)}{q^2-1}-2$, and so the bound to be met is
$$ \frac{(q^n-1)(q^{n-1}+q)-2(q^2-1)}{q^2-1}  \leq 2n \frac{q^{n+1}-q+1}{q-1}.$$
This holds if and only if $n=4$ and $q=4,5,7,8$, and in each case knowledge of $\alpha(g)$ and 
$|g|$ yields a contradiction by Lemma \ref{uu3}.

Finally, let us deal with the case when $L=\Omega_8^+(2)$. Here $|g|\in \{ 3, 4, 5, 7, 8, 9 \}$ and
$\dim \phi\geq 8$ (see Lemma \ref{md6}). (Indeed, $L$ does have a (unique) projective representation of degree $8$ 
for any characteristic $\ell$, and this is not an ordinary representation of $L$).

Suppose first that $g \in \SO_8^+(2)$. If $|g|=3$, then $\alpha(g)\leq 4$; if $|g|=4$, then $\alpha(g)\leq 3$; if $|g|\geq 5$,
then $\alpha(g)=2$. Since either $\dim \phi=8$ or $\dim \phi \geq 28$, applying the usual bound we are only left to consider 
the case when $\dim \phi=8$. 

Let $G=L$. If $\ell \neq p$, then using the GAP package we find that $\phi(g)$ is almost cyclic if and only if $g$ is one of 
elements described in the statement.

Now, suppose that $\ell = p = 7$.  Since the Sylow $7$-subgroups are cyclic, for $\ell = 7$ we can refer to 
\cite[Lemma 2.13 and Corollary 2.14] {DMZ12}, and we find that $\phi(g)$ is almost cyclic.

Next, suppose that $\ell=p=5$.
We observe that $L=\Omega_8^+(2)$ contains three classes of elements of order $5$ and three classes of maximal subgroups 
isomorphic to $\Sp_6(2)$.
Each of these symplectic groups contains elements from a single conjugacy class of elements $g\in L$  of order $5$ and
each element of order $5$ of $L$ is contained in a symplectic group.
(Note that the same holds for elements of order $9$, considering the three classes $9a$, $9b$, $9c$ of elements of 
order $9$ and the three classes of maximal subgroups isomorphic to $A_9$.)

Now, considering the group $2.L$, we see that the elements of the class $5a$ belong to a maximal subgroup $M$ of type 
$2\times \Sp_6(2)$ (a single class).
The restriction of $\phi$ to $M$ decomposes as the sum of a representation of degree $1$ and a representation $\psi$ of degree $7$.
Since $\psi(g)$ is almost cyclic (see Lemma \ref{SP2}), it follows that $\phi(g)$ is almost cyclic.

On the contrary, the elements of the classes $5b$ and $5c$ belong to a maximal subgroup $N$ of type $2.\Sp_6(2)$ 
(non-split central extensions: there are two classes of such subgroups).
The representation $\phi$ of degree $8$ of $2.L$ restricts irreducibly to $N$ and, working in $N$, we see that
$\phi(g)$ is not almost cyclic.

Finally, suppose that $\ell=p=3$.
The group $L$ has five classes $3a$, $3b$, $3c$, $3d$, $3e$ of elements of order $3$, and three classes $9a$, $9b$, 
$9c$ of elements of order $9$. 
Each of the classes $3a$, $3b$, $3c$, and similarly each of the classes $9a$, $9b$, $9c$ is contained in a unique 
maximal subgroup isomorphic to $\Sp_6(2)$, and conversely each of these symplectic groups contains elements from a 
single conjugacy class of elements of order $3$ and $9$ of $L$, respectively.

On the other hand, if $g$ belongs to one of the classes $3d$, $3e$, then $\alpha(g)=3$. This yields a contradiction, 
by Lemma \ref{uu3}.

Next, considering the group $2.L$, we see that
the elements of classes $3b, 3c, 9b, 9c$ belong to a maximal subgroup $M$ of type $2.\Sp_6(2)$
(there are two classes of such subgroups). The representation $\phi$ of degree $8$ of $2.L$ restricts irreducibly to $M$ and, 
working in $M$,
we obtain that $\phi(g)$ is almost cyclic if and only if $g$ has order $9$.

On the other hand, the elements $g$ of the classes $3a, 9a$ belong to a maximal subgroup $N$ of type $2\times \Sp_6(2)$ 
(a single class). The restriction of $\phi$ to such subgroups decomposes as the sum of a representation of degree $1$ and 
a representation $\psi$ of degree $7$. Since $\psi(g)$ is almost cyclic, $\phi(g)$ is also almost cyclic. 

Next, let $G=\langle L, g \rangle= \SO_8^+(2)$. Then $p=2$, and using the GAP package one can check that
$\phi(g)$ is almost cyclic if
and only if either $|g| \in 4f$ or $|g|=8$, as claimed.

Finally, suppose that $g \notin \SO_8^+(2)$, and hence $|g|=3,9$. Then $\alpha(g)\leq 8$; moreover, either  
$\dim \phi\leq 50 $ or $\dim \phi\geq 105$ (see \cite{Atl}  and \cite{MAtl}).
If $\ell\neq 3$, then use of the GAP package shows that $\phi(g)$ is not almost cyclic.
If $\ell=3$,  we need to consider representations of degree $28$ or $48$.
Applying the usual bound, we can easily see that $\phi(g)$ is not almost cyclic if $|g|=3$.
Now, let $g \not \in L$ be an element of order $9$ belonging so a class $C$. Then the power maps in \cite{Atl} 
show that $g^3$ belongs to the class $3d$ of $L$, and therefore  every element $x\in 3d$ is the cube of some $y \in C$.
We find that we can choose three elements in $3d$, say $x_1,x_2,x_3$, which generate $L$. Let $g_1,g_2,g_3$ 
three elements of $C$  such that
$x_i=g_i^3$ for $i=1,2,3$. 
Thus, the group $H=\langle g_1,g_2,g_3 \rangle$ contains $\langle g_1^3, g_2^3, g_3^3\rangle =L$, whence $H=L.3$.

The above implies that $\alpha(g)\leq 3$ for all $g \in C$. Thus, by Lemma \ref{uu3}, $\phi(g)$ cannot be almost cyclic.
\end{proof}

\section{The exceptional groups of Lie type}\label{sec exc}

In this section we deal with the groups $\Aut L$, where $L$ is a finite simple exceptional group of Lie type.
The notation for the groups $L$ is standard. Note however that our notation for $q$ in the case of twisted groups is 
such that ${}^2E_6(q)\subset E_6(q^2)$, ${}^3D_4(q)\subset D_4(q^3)$,
 ${}^2B_2(q)\subset B_2(q)$, ${}^2F_4(q)\subset F_4(q)$ and ${}^2G_2(q)\subset G_2(q)$. 
We also recall that the structure of $\Aut L$ 
 can be found in \cite[Table 5.1B]{KL}. For a thorough and detailed reference, however, see \cite[Section 2.5 (Theorem 2.5.12)]{GLS}.

The maximal order of an element in $\Aut L$ can be easily deduced from \cite{GMPS,KS}:
these data are summarised in Table \ref{tab1}.  Lower bounds for the dimension of a non-trivial 
cross-characteristic projective representation of $L$ are found in \cite[Table 1, p. 14]{Ho}.
Finally, we recall that, for any non-trivial element $g\in \Aut L$, the minimal number $\alpha(g)$ of $L$-conjugates of 
$g$ sufficient to generate $\langle g,L \rangle$ is provided by Proposition  \ref{GS1}(2).

\begin{table}[!th]
\begin{center}
\begin{tabular}{ccc}
$L$ & Conditions & $|g|\leq $\\\hline\\[-8pt]
${}^2B _2(2^{2e+1})$ & $e\geq 1$ & $(2e+1)(2^{2e+1}+2^{e+1}+1)$\\
$G_2(q)$ & $q=r^e$, $r\neq 3$ & $ e (q^2+q+1)$\\
$G_2(q)$ & $q=3^e$ & $ 2e (q^2+q+1)$\\
${}^2G_2(3^{2e+1})$ & $e\geq 1$ & $(2e+1)(3^{2e+1}+3^{e+1}+1)$\\  
${}^3D_4(q)$ & $q=r^e$ & $3e(q^3-1)(q+1)$\\
$F_4(q)$ &  & $q(q^3-1)(q+1)$\\
${}^2F_4(2^{2e+1})$ & $e\geq 1$ & $(2e+1)(2^{4e+2}+2^{3e+2}+ 2^{2e+1}+2^{e+1} +1)$\\
$E_6(q)$ &  & $q^2(q^3+1)(q^2+q+1)$\\ 
${}^2E_6(q)$ &   & $q (q+1)(q^2+1)(q^3-1)$\\
$E_7(q)$ &  & $q(q+1)(q^2+1)(q^4+1)$ \\
$E_8(q)$ &  & $q(q+1)(q^2+q+1)(q^5-1)$\\[3pt]
\end{tabular}
\end{center}

\caption{Maximal order of an element in $\Aut L$.}\label{tab1}
\end{table}

\begin{lemma}\label{except}
Let $L$ be a simple exceptional group of Lie type.
Let $1\neq g\in \Aut L$ be a $p$-element and let  $G=\langle g,L\rangle$.
Then $g$ is almost cyclic in some cross-characteristic projective irreducible representation $\phi$ of $G$
that is non-trivial on $L$ if and only if one if one of the cases 
listed in Table \ref{exctbl} below occurs.
\end{lemma}

\begin{proof}
Suppose that $\phi(g)$ is almost cyclic.

If $L={}^2B _2(2^{2e+1})$, where $e\geq 2$, then $\alpha(g)\leq 5$.
By Lemma \ref{uu3} and Tables \ref{tab1} and \cite[Table 1, p. 14]{Ho}, the bound 
$2^e(2^{2e+1}-1)\leq 5  \left( (2e+1)(2^{2e+1}+2^{e+1}+1) -1\right)$ must be met.
However, the bound only holds for $e\leq 6$. So, we may assume $2 \leq e \leq 6$.
If $g \in L$, then $|g| \leq 2^{2e+1}+2^{e+1}+1$, and hence we can refine the bound to
$2^e(2^{2e+1}-1)\leq 5 ( 2^{2e+1}+2^{e+1})$. However, this only holds for 
$e=2$. This means that for $e=3,4,5,6$, it suffices to examine elements
 $g \not \in L$. But, for these values of $e$, $|g|\leq 2e+1$, which yields $2^e(2^{2e+1}-1)> 10e$, a contradiction.

If $L={}^2B _2(2^3)$, then $|g|\in\{  2, 3, 4, 5, 7, 13 \}$ and $\dim \phi\geq 8$.
If $|g|=2$, $\alpha(g)=3$, whereas, if $|g|>2$, $\alpha(g)=2$.
We readily get a contradiction if $|g|=2,3,4$. So, we may assume $|g|\geq 5$, which implies $g\in L$.
Inspection of the Brauer characters shows that we are left to consider the following cases:
(i) $|g|=5,7,13$, $\ell=5$ and $\dim \phi=8$;
(ii) $|g|=13$, $\ell\neq 2$ and $\dim \phi=14$; 
(iii) $|g|=13$, $\ell=13$ and $\dim \phi=16,24$.
Using the GAP package we obtain item (1) of the statement.
If $L={}^2B _2(2^5)$, then $|g|\in \{ 2, 4, 5, 25, 31, 41 \}$ and $\dim \phi\geq 124$.
Since $\alpha(g)\leq3$, we get a contradiction.

Next, suppose that $L=G_2(q)$.  If  $q\geq 5$, then a detailed check shows that we get a contradiction, by Lemma \ref{uu3}.

Next, let $L=G_2(3)$. Then $|g| \in \{  2, 3, 4, 7, 8, 9, 13 \}$ and 
either $\dim \phi= 14$ or $\dim \phi \geq 27$.
If $|g|\leq 3$, then $\alpha(g)\leq 4$, otherwise $\alpha(g)=2$.
Thus we readily get a contradiction, unless $|g|=8,9,13$ and $\dim \phi=14$.
Using the GAP package, we find that $\phi(g)$ is almost cyclic only for $|g|=13$.

Finally, let $L=G_2(4)$. Then $|g|\in \{ 2, 3, 4, 5, 7, 8, 13, 16 \}$ and  either 
$\dim \phi=12$ or $\dim \phi \geq 64$.
If $|g|\geq 4$, then $\alpha(g)=2$, whereas $\alpha(g)\leq 4$ if $|g|\leq 3$.
We readily get a contradiction, unless $\dim \phi=12$ and $|g|\geq 7$.
Using the GAP package, we find  that $\phi(g)$ is almost cyclic 
if and only if $|g|=13,16$.

Now, suppose that $L={}^3D_4(q)$, where $q=r^e$. 
Ruling out this group requires a rather delicate analysis, which we need to develop in detail.
Here $\alpha(g)\leq 7$ and we have to consider the bound 
$r^{3e}(r^{2e}-1)\leq 7 (3e(r^{4e}+r^{3e}-r^e-1) -1)$, which
only holds when $q$ has one of the following values:
$q=2^e$, where $e\leq 7$; $q=3^e$, where $e\leq 4$;
$q=5^e$, where $e\leq 2$; $q=7,11,13,17,19$.

Suppose first that $g$ is semisimple. From the description of the maximal tori of $L$ given in \cite{Kl3D4}, 
we can extract the following table, which gives, for each value of $q$, an upper bound $s(q)$ for the maximal order of a 
semisimple
$p$-element of $\Aut L$:
\begin{center}
\begin{tabular}{c|c||c|c||c|c||c|c||c|c||c|c}
$q$ &  $s(q)$ & $q$ &  $s(q)$  & $q$ &  $s(q)$  & $q$ &  $s(q)$ & $q$ &  $s(q)$  & $q$ &  $s(q)$  \\\hline
$2$ & $13$ &
$3$ & $73$ &
$4$ & $241$ &
$5$ & $601$ &
$7$ & $181$ &
$8$ & $243$ \\
$9$ & $6481$ &
$11$ & $1117$ &
$13$ & $28393$ & 
$16$ & $673$   &
$17$ & $83233$ &
$19$ & $769$\\
$25$ & $390001$ &
$27$ & $530713$ &
$32$ & $1321$ &
$64$ & $38737$ &
$81$ & $6481$ &
$128$ & $14449$
\end{tabular}
\end{center}
It is easy to verify that for $q\geq 7$ we get a contradiction, by Lemma \ref{uu3}.

Next, suppose that $g$ is unipotent. Then $\eta_p(L)=p$ if $p\geq 3$, while $\eta_2(L)=8$ and $\eta_3(L)=9$ (see \cite{Tes}).
Thus we readily get a contradiction for $q\geq 3$.

Let $L={}^3D_4(2)$. Then $|g|\in \{ 2, 3, 4, 7, 8, 9, 13 \}$ and $\dim \phi\geq 25$.
If $|g|\geq 4$, then $\alpha(g)=2$, whereas, if $|g|=2,3$, then $\alpha(g)\leq 4$. In both cases, we get a 
contradiction.

At this stage, we are left to examine the cases  $q=3,4,5$ and $g$ semisimple.

Let $L={}^3D_4(3)$. Then $|g|\in \{2, 4, 7, 8, 13, 73 \}$. 
Since $\dim \phi\geq 216$, we get a contradiction unless $|g|=73$.
In this case, using the GAP package, we find that $\alpha(g)=2$, which yields a contradiction.

Let $L={}^3D_4(4)$. Then $|g|\in \{ 3, 5, 7, 9, 13, 27, 241 \}$.
Since $\dim \phi\geq 960$, we get a contradiction unless $|g|=241$. Finally, let $L={}^3D_4(5)$. 
Then $|g|\in \{ 2, 3, 4, 7, 8, 9, 31, 601  \}$.
Since $\dim \phi\geq 3000$, we get a contradiction unless $|g|=601$.

Notice that in the instances $|g|=241$ for $L={}^3D_4(4)$, and $|g|=601$ for $L={}^3D_4(5)$, the exceptional element $g$ 
belongs to $L$ and is a generator of the cyclic torus
$T_5$ (in the notation of \cite{Kl3D4}). Since $g$ is regular semisimple, by Lemma \ref{g12}  we have $\alpha(g)\leq 3$,
whence a contradiction.

Finally,  we list in  Table \ref{exctbl} the exact occurrences of almost cyclic elements
for the simple group ${}^2F_4(2)'$,
the so-called Tits group, as well as for the simple groups $G=G_2(2)'$ and $G={}^2G_2(3)'$.
These results are based, as usual, on Lemma \ref{uu3} and on direct computations with GAP.
All the remaining exceptional groups are ruled out using Lemma \ref{uu3}. The relevant computations are left to the reader.
\end{proof}

\begin{table}[h]
$$\begin{array}{c|c|c|c||c|c|c|c}
G & \ell & \dim \phi & |g| & G & \ell & \dim \phi & |g|\\\hline
{}^2B _2(2^3) & \textrm{all} & 14 & 13  & {}^2B _2(2^3) & 5 & 8 & 7,13\\
G_2(2)' & \textrm{all} &  6,7 & 7,8 & G_2(2)' & 3 & 3 & \textrm{all}  \\
G_2(2) & \textrm{all} & 6,7 & 8 &  G_2(3) & \textrm{all} & 14 & 13\\
G_2(4)  & \textrm{all} &  12 &  13 & G_2(4).2  & \textrm{all} &  12 &  16\\
 {}^2G_2(3)'  & \neq 2 &  7,8 & 7,9 & {}^2G_2(3)'  & \neq 2,7 & 9 & 9 \\
{}^2G_2(3)'& 2 & 2 & \textrm{all} & {}^2G_2(3)'& 2 & 4  & 3,7,9\\ 
 {}^2G_2(3)'& 2 & 8 & 7,9 & 
 {}^2G_2(3) & \neq 2 & 7,8 & 9\\
 {}^2G_2(3) & 2 & 6,8 & 9
  \end{array}$$
\caption{Occurrences of almost cyclic elements in exceptional groups of Lie type.}\label{exctbl}
\end{table}

\section{Proof of the Main Theorem}\label{sec proof}

The following theorem answers with full details the question raised in the present paper, 
thus also proving the Main Theorem quoted in the Introduction.

\begin{theo}\label{main2}
Let $L$ be a finite simple group of Lie type, let $g \in \Aut L$ be a $p$-element for some prime $p$, and let 
$G=\langle L,g \rangle$. Let $\phi$ be an irreducible projective representation of $G$ over an algebraically closed 
field $\F$ of characteristic $\ell$, different from the defining characteristic of $L$. Suppose that $\phi$ is 
non-trivial on $L$; furthermore, assume that $g^2\neq 
1$, $L\neq \PSL_2(q)$ and $\phi(g)$ is almost cyclic. Then, either
$\phi_{|L}$ is an irreducible Weil representation of $L$, where $L\in \{\PSL_n(q)\;  (n\geq 3),\; \PSU_n(q) \; (n\geq 
3),\; \PSp_{2n}(q)\; 
(n\geq 2,\, q \textrm{ odd} )\}$ and $g$ is semisimple, or 
$L$ is of exceptional type and $(\dim \phi, \ell, |g|)$ are given in {\rm Table \ref{exctbl}}, or
one of the following occurs:
\begin{itemize}
\item[(1)] $L=\PSL_3(2)$, $\dim \phi=2$,   $\ell=7$ and either $|g|=3,4,7$ or $g\notin L$ and $|g|=8$;
\item[(2)] $L=\PSL_3(2)$, $\dim \phi=3,4$ and  and either $|g|=3,4,7$ or $g\notin L$ and $|g|=8$;
\item[(3)] $L=\PSL_3(2)$, $\dim \phi=5$, $\ell=7$ and either $|g|=4$ or $g\notin L$ and $|g|=8$;
\item[(4)] $G=\PSL_3(2)$, $\dim \phi=6$, $\ell=7$ and  $|g|=7$;
\item[(5)] $G=\PSL_3(2).2$, $\dim \phi=6,7$, and  $|g|=8$, with $g\notin L$;
\item[(6)] $L=\PSL_3(2)$, $\dim \phi=8$  and either $\ell \neq 3,7$ and $|g|=7$, or $\ell\neq 7$, $g\notin L$ and 
$|g|=8$;
\item[(7)] $G=\PSL_3(4)$, $\dim \phi=4$,  $\ell=3$ and $|g|=3,4,5,7$;
\item[(8)] $G=\PSL_3(4)$, $\dim \phi=6$ and $|g|=5,7$;
\item[(9)] $G=\PSL_3(4)$, $\dim \phi=8$ and $|g|=7$;
\item[(10)] $G=\PSL_3(4).2_1$, $\dim\phi=6$ and $|g|=8$;
\item[(11)] $G=\PSL_3(4).2_2$, $\dim\phi=4,6$, $\ell=3 $ and $|g|=8$;
\item[(12)] $G=\PSL_3(4).2_3$, $\dim\phi=6$, $\ell=3$ and $|g|=8$;
\item[(13)] $G=\PSL_3(4).2_3$, $\dim\phi=8$  and $|g|=8$;
\item[(14)] $G=\PSL_3(4).2_3$, $\dim\phi=10$, $\ell=3$ and $g\in 8a$;
\item[(15)] $G=\PSL_4(2)$, $\dim \phi=7$ and either $g\in 3a$ or $|g|=5,7$;
\item[(16)] $G=\PSL_4(2)$, $\dim \phi=8$ and $|g|=7$;
\item[(17)] $G=\PSL_4(2).2$, $\dim \phi=7$ and $|g|=4,8$;
\item[(18)] $G=\PSL_4(2).2$, $\dim \phi=8$ and $|g|=8$;
\item[(19)] $G=\PSU_3(3).2$, $\dim\phi=6$ and $|g|=8$;
\item[(20)] $G=\PSU_3(3).2$, $\dim\phi=7$, $\ell\neq 2$ and $|g|=8$;
\item[(21)] $G=\PSU_3(4)$, $\dim\phi=12$ and $|g|=16$;
\item[(22)] $G=\PSU_4(2)$, $\dim\phi =4$, $\ell\neq 3$ and $g \in 3a,3b,3d$ or $|g|=4,5,9$;
\item[(23)] $G=\PSU_4(2)$, $\dim\phi =4$, $\ell = 3$ and $g \in 3a,3b,4a$ or $|g|=5,9$;
\item[(24)] $G=\PSU_4(2)$, $\dim\phi =5$ and $g \in 4b$;
\item[(25)] $G=\PSU_4(2).2$, $\dim\phi =4$, $\ell = 3$ and $g \in 4d$ or $|g|=8$;
\item[(26)] $G=\PSU_4(2).2$, $\dim\phi =5$, $\ell = 3$ and  $|g|=4,8$;
\item[(27)] $G=\PSU_4(2).2$, $\dim\phi =6$, $\ell \neq 3$ and $|g|=4,8$;
\item[(28)] $G=\PSp_4(3)$, $\dim\phi =4$ and $g \in 3a,3b,3d$ or $|g|=9$;
\item[(29)] $G=\PSp_4(3)$, $\dim\phi =5$, $\ell \neq 2$ and $g \in 3d$ or $|g|=9$;
\item[(30)] $G=\PSp_4(3)$, $\dim\phi =6$ and $g \in 3c$ or $|g|=5,9$;
\item[(31)] $G=\PSp_4(3).2$, $\dim\phi =6$ and $|g|=4,8$;
\item[(32)] $G=\PSU_4(3)$, $\dim \phi=6$ and either $g \in 3b$ or $|g|=5,7,8,9$;
\item[(33)] $G=\PSU_4(3).2_2$, $\dim \phi=6$ and either $g \in 4d$ or $|g|=8$;
\item[(34)] $G=\PSU_5(2).2$, $\dim \phi=10$ and $|g|=16$;
\item[(35)] $G=\Sp_4(2)'$, $\dim \phi=3,4,5$ and $\ell\neq 2,3$;
\item[(36)] $G=\Sp_4(2)'$, $\dim \phi=6$, $\ell\neq 2$ and $|g|=5$;
\item[(37)] $G=\Sp_4(2)'$, $\dim \phi=3,4$ and $\ell=2$;
\item[(38)] $G=\Sp_4(2)'$, $\dim \phi=2,3,4$ and $\ell=3$;
\item[(39)] $G=\PSp_4(4)$, $\dim \phi=18$ and $|g|=17$;
\item[(40)] $G=\PSp_6(2)$, $\dim \phi=7$ and either $g \in 3a,4c$ or $|g|=5,7,8,9$;
\item[(41)] $G=\PSp_6(2)$, $\dim \phi=8$ and $|g|=7,8,9$;
\item[(42)] $G=\POmega_8^+(2)$, $\dim \phi=8$ and either $g \in 3a, 5a, 8b$ or $|g|=7,9$;
\item[(43)] $G=\PSO_8^+(2)$, $\dim \phi=8$ and either $g \in 4f$ or $|g|=8$.
\end{itemize}
Moreover, in each of the cases listed above, $\phi(g)$ is indeed almost cyclic.
\end{theo} 

\begin{proof}
Suppose that $L$ is a classical group, and assume furthermore that $L$ is neither orthogonal nor symplectic of even 
characteristic. By Lemma \ref{3hh}, we may assume that $p$ divides $|L|$.

Assume first that $g$ is a semisimple element. Observe that, if $g \notin A_d$, then $\phi(g)$ is 
not almost cyclic, unless $L \in \{\PSU_3(3), \PSU_4(3), \PSp_4(3)\}$ (see Lemmas \ref{p67b}, \ref{newp2p}, 
\ref{e24}, \ref{5.7}, \ref{nn5u}, \ref{2ss} and \ref{4ss}). Now, suppose that $g \in A_d$. Then (see the argument 
following Lemma \ref{3hh}) either $\phi_{|L}$ has an irreducible constituent which is not a Weil representation, or  
$\phi_{|L}$ is an irreducible Weil representation of $L$. It follows that $\phi(g)$ is not almost cyclic, unless either  
$p>2$ and $L \in \{\PSL_3(2), \PSL_3(4), \PSL_4(2)\}$ (see Lemmas \ref{nonW} and  \ref{newp2p}), or $L \in \{\PSU_4(2), 
\PSU_4(3)\}$ or $L= \PSp_4(3)$ (see Lemmas  \ref{nonW} and \ref{4ss}). 
 
Next,
assume that $g$ is unipotent, without any restriction to 
$\phi$. Then $\phi(g)$ is not almost cyclic unless $L \in \{\PSL_3(2), \PSL_3(4), \PSL_4(2), \PSp_4(3), \PSU_5(2)\}$ 
(see 
Lemmas \ref{LU2}, \ref{SL34}, \ref{u34} and \ref{SP43}). The above listed  exceptions are analysed in Section \ref{low}.
The case $L=\PSL_3(2)$ is dealt with in  \ref{L32}, yielding items (1) to (6) of the statement. 
The case $L=\PSL_3(4)$ in \ref{L34}, yielding items (7) to (14) of the statement. 
The case $L=\PSL_4(2)$ is dealt with in \ref{L42}, yielding items (15) to (18) of the statement. 
The case $L=\PSU_3(3)$ is dealt with in \ref{U33}, yielding items (19) and (20). 
The case $L=\PSU_3(4)$ is dealt with in \ref{U34}, where it is found that $\phi(g)$ is never almost cyclic, unless
$\phi$ is a Weil representation of dimension $12$ and $|g|=16$, thus yielding item (21). 
The case $L=\PSU_4(2) \cong \PSp_4(3)$ needs some extra care. 
If we view $L$ as $\PSU_4(2)$, the Weil representations of $L$ have dimension $5$ and $6$.  
Thus (see \ref{U42}), we get items (22) to (27). 
If we view $L$ as $\PSp_4(3)$, the Weil representations have dimension $4$ and $5$. 
This group is dealt with in Lemma \ref{SP43} if $g$ is unipotent (that is, a $3$-element), and in \ref{U42} if $g$ is
semisimple, yielding items (28) to (31). 
The case $L = \PSU_4(3)$ is dealt with in \ref{U43}, yielding items (32) and (33). 
The case $L = \PSU_5(2)$ is dealt with in \ref{LU2}, yielding item (34) of the statement.

Finally, assume that  $L$ is a classical group, either orthogonal or symplectic of even characteristic.  These groups are
dealt with in Section 
\ref{otherclass}. If $L$ is symplectic, then we get items (35) to (41) of the statement, by Lemma \ref{SP2} and Remark \ref{Sp42d}.
Lemmas \ref{O1} and \ref{O3} deal with the orthogonal groups, yielding items (42) and (43) of the statement.
\end{proof}

\section*{Appendix}

\noindent Here some useful routines for MAGMA, kindly provided by the anonymous referee.

\begin{verbatim}
function IsAlmostCyclic(A)
 f:=CharacteristicPolynomial(A) div MinimalPolynomial(A);
 if f eq 1 or 
  (#Factorization(f) eq 1 and Factorization(f)[1,2] eq Degree(f)) 
  then return true;
 end if;
 return false;
end function;

function AllAlmostCyclic(M)
 AAC:=[]; 
 rho:=Representation(M);
 for c in Remove(ConjugacyClasses(Group(M)),1) do
  if (not(c[1] eq 2) and IsPrimePower(c[1])
   and not(IsDivisibleBy(c[1],Characteristic(Field(M))))) then
   A:=rho(c[3]);
   if(IsAlmostCyclic(A)) then Append(~AAC,c); end if;
  end if;
 end for;
return AAC;
end function;

// Let G be a permutation group in characteristic p.

function FindAllCases(G,p)
 AllCases:=[];
 AllCasesReps:=[];
 ps:=[i[1]:i in FactoredOrder(G)|i[1] ne p];
 for l in ps do
  _,Irr:=IrreducibleModules(G,GF(l));
  Remove(~Irr,1); // remove the trivial.
  for M in Irr do
   AAC:=AllAlmostCyclic(M);
   for i in AAC do
    Append(~AllCases,[l,Dimension(M),i[1]]);
    Append(~AllCasesReps,<M,i[3]>);
   end for;
  end for;
 end for;
return AllCases,AllCasesReps;
end function;
\end{verbatim}

\end{document}